\documentclass[aap,MSNbibl,dvips]{arximspdf}
\usepackage{graphicx}
\doi{10.1214/12-AAP903} 
\volume{23}
\issue{6}
\pubyear{2013}
\firstpage{2327}
\lastpage{2356}

\makeatletter
\newcommand{\rrvert}{\vert}
\newcommand{\llvert}{\vert}

\newtheorem{theorem}{Theorem}[section]
\newtheorem{lem}[theorem]{Lemma}
\newtheorem{prop}[theorem]{Proposition}
\newtheorem{cor}[theorem]{Corollary}

\newproclaim{rem}[theorem]{Remark}

\newcommand{\R}{\mathbb{R}}
\newcommand{\N}{\mathbb{N}}
\newcommand{\E}{\mathbb{E}}

\renewcommand{\P}{\mathbb{P}}

\makeatother

\begin{document}
\begin{frontmatter}

\title{Optimal stopping problems for the maximum process~with upper
and lower caps}
\runtitle{Optimal stopping for the maximum process}

\begin{aug}
\author[A]{\fnms{Curdin} \snm{Ott}\corref{}\ead[label=e1]{C.Ott@bath.ac.uk}}
\runauthor{C. Ott}
\affiliation{University of Bath}
\address[A]{Department of Mathematical Sciences\\
University of Bath\\
Claverton Down\\
Bath BA2 7AY\\
United Kingdom\\
\printead{e1}} 
\end{aug}

\received{\smonth{7} \syear{2011}}
\revised{\smonth{10} \syear{2012}}

%
\begin{abstract}
This paper concerns optimal stopping problems driven by the running
maximum of a spectrally negative L\'evy process $X$. More precisely, we
are interested in modifications of the Shepp--Shiryaev optimal stopping
problem [Avram, Kyprianou and Pistorius \textit{Ann. Appl. Probab.}
\textbf{14} (2004) \mbox{215--238}; Shepp and Shiryaev \textit{Ann. Appl.
Probab.} \textbf{3} (1993) 631--640; Shepp and Shiryaev \textit{Theory
Probab. Appl.} \textbf{39} (1993) 103--119]. First, we consider a
capped version of the Shepp--Shiryaev optimal stopping problem and
provide the solution explicitly in terms of scale functions. In
particular, the optimal stopping boundary is characterised by an
ordinary differential equation involving scale functions and changes
according to the path variation of $X$. Secondly, in the spirit of
[Shepp, Shiryaev and Sulem \textit{Advances in Finance and Stochastics}
(2002) 271--284 Springer], we consider a modification of the capped
version of the Shepp--Shiryaev optimal stopping problem in the sense
that the decision to stop has to be made before the process $X$ falls
below a given level.
\end{abstract}

%

\pdfkeywords{60G40, 60G51, 60J75, Optimal stopping,
optimal stopping boundary, principle of smooth fit,
principle of continuous fit, Levy processes,
scale functions}

\end{frontmatter}

\section{Introduction}
Let $X=\{X_t\dvtx t\geq0\}$ be a spectrally negative L\'evy process defined
on a filtered probability space $(\Omega,\mathcal{F},\mathbb{F}=\{
\mathcal{F}_t\}_{t\geq0},\P)$ satisfying the natural conditions;
cf. \cite{bichteler}, Section 1.3, page 39. For $x\in\R$, denote by
$\P_x$ the probability measure under which $X$ starts at $x$ and for
simplicity write $\P_0=\P$. We associate with $X$ the maximum process
$\overline X=\{\overline X_t\dvtx t\geq0\}$ given by $\overline X_t:=s\vee
\sup_{0\leq u\leq t} X_u$ for $t\geq0,s\geq x$. The law under which
$(X,\overline X)$ starts at $(x,s)$ is denoted by $\P_{x,s}$.

In this paper we are mainly interested in the following optimal
stopping problem:
%
\begin{equation}\label{problem1}
V^*_\epsilon(x,s)=\sup_{\tau\in\mathcal{M}}\E_{x,s}
\bigl[e^{-q\tau+\overline X_\tau\wedge\epsilon} \bigr],
\end{equation}
where $\epsilon\in\R,q>0,(x,s)\in \mathbb R^2 := \{(x_1, s_1) \in E |
x_1 \leq s_1\}$, and $\mathcal{M}$ is the set of all finite $\mathbb
{F}$-stopping times. Since the constant $\epsilon$ bounds the process
$\overline X$ from above, we refer to it as the upper cap. Due to the
fact that the pair $(X,\overline X)$ is a strong Markov process, (\ref
{problem1}) has also a Markovian structure and hence the general theory
of optimal stopping \cite{peskir} suggests that the optimal stopping
time is the first entry time of the process $(X,\overline X)$ into some
subset of $E$. Indeed, it turns out that under some assumptions on $q$
and $\psi(1)$, where $\psi$ is the Laplace exponent of $X$ [see (\ref
{Lexponent}), page \pageref{Lexponent}, for a formal definition], the
solution of (\ref{problem1}) is given by
\[
\tau_\epsilon^*:=\inf\bigl\{t\geq0\dvtx \overline X_t-X_t
\geq g_\epsilon(\overline X_t)\bigr\}
\]
for some function $g_\epsilon$ which is characterised as a solution to
a certain ordinary differential equation involving scale functions. The
function $s\mapsto s-g_\epsilon(s)$ is sometimes referred to as the
optimal stopping boundary. We will show that the shape of the optimal
boundary has different characteristics according to the path variation
of $X$. The solution of problem (\ref{problem1}) is closely related to
the solution of the Shepp--Shiryaev optimal stopping problem
%
\begin{equation}\label{problem2}
V^*(x,s)=\sup_{\tau\in\mathcal{M}}\E_{x,s} \bigl[e^{-q\tau
+\overline X_\tau}
\bigr],
\end{equation}
which was first studied by Shepp and
Shiryaev \cite{russianoption,anewlook} for the case when $X$ is a
linear Brownian motion and later by Avram, Kyprianou and
Pistorius \cite{exitproblems} for the case when $X$ is a spectrally
negative L\'evy process. Shepp and Shiryaev \cite{russianoption}
introduced the problem as a means to pricing Russian options. In the
latter context the solution of (\ref{problem2}) can be viewed as the
fair price of such an option. If we introduce a cap $\epsilon$, an
analogous interpretation of the solution of (\ref{problem1}) applies,
but for a Russian option whose payoff was moderated by capping it at a
certain level (a fuller description is given in
Section \ref{application}).

Our method for
solving (\ref{problem1}) consists of a verification technique, that is,
we heuristically derive a candidate solution and then verify that it is
indeed a solution. In particular, we will make use of the principle of
smooth and continuous fit \cite{peskir,mikhalevich,pesshir,someremarks}
in a similar way
to \cite{russianoption,maximalityprinciple}.

It is also natural to ask for a modification of (\ref{problem1}) with
a lower cap. Whilst this is already included in the starting point of
the maximum process $\overline X$, there is a stopping problem that
captures this idea of lower cap in the sense that the decision to
exercise has to be made before $X$ drops below a certain level.
Specifically, consider
%
\begin{equation}\label{problem3}
V^*_{\epsilon_1,\epsilon_2}(x,s)=\sup_{\tau\in\mathcal
{M}_{\epsilon_1}}\E_{x,s}
\bigl[e^{-q\tau+\overline X_\tau\wedge
\epsilon_2} \bigr],
\end{equation}
where $\epsilon_1,\epsilon_2\in\R$ such that $\epsilon_1<\epsilon
_2$, $q>0,\mathcal{M}_{\epsilon_1}:=\{\tau\in\mathcal{M} \vert
\tau\leq T_{\epsilon_1}\}$ and $T_{\epsilon_1}:=\break\inf\{t\geq
0\dvtx X_t\leq\epsilon_1\}$.
In the special case of no cap ($\epsilon_2=\infty$), this problem was
considered by Shepp, Shiryaev and Sulem \cite{abarrierversion} for
the case where $X$ is a linear Brownian motion. Inspired by their
result we expect the optimal stopping time to be of the form
$T_{\epsilon_1}\wedge\tau^*_{\epsilon_2}$, where $\tau^*_{\epsilon
_2}$ is the optimal stopping time in (\ref{problem1}). Our main
contribution here is that, with the help of excursion theory (cf. \cite
{kyprianou,bertoinbook}), we find a closed form expression for the
value function associated with the strategy $T_{\epsilon_1}\wedge\tau
^*_{\epsilon_2}$, thereby allowing us to verify that it is indeed an
optimal strategy.

This paper is organised as follows. In Section \ref{application} we
provide some motivation for studying (\ref{problem1}) and (\ref
{problem3}). Then we introduce some more notation and collect some
auxiliary results in Section \ref{not}. Our main results are presented
in Section \ref{mainre}, followed by their proofs in Sections \ref
{firsttheorem} and \ref{mainrepf}. Finally, some numerical examples
are given in Section \ref{applications}.

\section{Application to pricing capped Russian options}\label{application}
The aim of this section is to give some motivation for studying (\ref
{problem1}) and (\ref{problem3}).

Consider a financial market consisting of a riskless bond and a
risky asset. The value of the bond $B=\{B_t\dvtx t\geq0\}$ evolves
deterministically such that
%
\begin{equation}\label{discount}
B_t=B_0e^{rt},\qquad B_0>0,r\geq0,t
\geq0.
\end{equation}
The price of the risky asset is modeled as the exponential spectrally
negative L\'evy process
%
\begin{equation}\label{model}
S_t=S_0e^{X_t},\qquad S_0>0,t\geq0.
\end{equation}
In order to guarantee that our model is free of arbitrage we will
assume that $\psi(1)=r$. If $X_t=\mu t+\sigma W_t$, where $W=\{
W_t\dvtx t\geq0\}$ is a standard Brownian motion, we get the standard
Black--Scholes model for the price of the asset. Extensive empirical
research has shown that this (Gaussian) model is not capable of
capturing certain features (such as skewness and heavy tails) which are
commonly encountered in financial data, for example, returns on stocks.
To accommodate for these problems, an idea, going back to \cite
{merton}, is to replace the Brownian motion as the model for the
log-price by a general L\'evy process $X$; cf. \cite{chan}. Here we
will restrict ourselves to the model where $X$ is given by a spectrally
negative L\'evy process. This restriction is mainly motivated by
analytical tractability. It is worth mentioning, however, that Carr and
Wu \cite{carr} as well as Madan and Schoutens \cite{madan} have
offered empirical evidence to support the case of a model in which the
risky asset is driven by a spectrally negative L\'evy process for
appropriate market scenarios.

A capped Russian option is an option which gives the holder the
right to exercise at any almost surely finite stopping time $\tau$
yielding payouts
\[
e^{-\alpha\tau} \Bigl(M_0\vee\sup_{0\leq u\leq\tau}S_u
\wedge C \Bigr),\qquad C> M_0\geq S_0,\alpha>0.
\]
The constant $M_0$ can be viewed as representing the ``starting''
maximum of the stock price (say, over some previous period $(-t_0,0])$.
The constant $C$ can be interpreted as cap and moderates the payoff of
the option. The value $C=\infty$ is also allowed and corresponds to no
moderation at all. In this case we just get the normal Russian option.
Finally, when $C=\infty$ it is necessary to choose $\alpha$ strictly
positive to guarantee that it is optimal to stop in finite time and
that the value is finite; cf. Proposition \ref{solution2}.

Standard theory of pricing American-type options \cite
{shirfin} directs one to solving the optimal stopping problem
%
\begin{equation}\label{motivation}
V_r(M_0,S_0,C):=B_0\sup
_{\tau}\E\Bigl[B^{-1}_\tau
e^{-\alpha\tau
} \Bigl(M_0\vee\sup_{0\leq u\leq\tau}S_u
\wedge C \Bigr) \Bigr],
\end{equation}
where the supremum is taken over all $[0,\infty)$-valued $\mathbb
{F}$-stopping times. In other words, we want to find a stopping time
which optimises the expected discounted claim. The right-hand side
of (\ref{motivation}) may be rewritten as
\[
V_r(M_0,S_0,C)=V^*_\epsilon(x,s)=
\sup_{\tau\in\mathcal{M}}\E_{x,s} \bigl[e^{-q\tau+\overline X_\tau
\wedge\epsilon}\bigr],
\]
where $q=r+\alpha,x=\log(S_0),s=\log(M_0)$ and $\epsilon=\log( C
)$.

In (\ref{motivation}) one might only allow stopping times that
are smaller or equal than the first time the risky asset $S$ drops
below a certain barrier. From a financial point of view this
corresponds to a default time after which all economic activity stops;
cf.~\cite{abarrierversion}. Including this additional feature leads
in an analogous way to the above optimal stopping problem (\ref{problem3}).


\section{Notation and auxiliary results}\label{not}
The purpose of this section is to introduce some notation and collect
some known results about spectrally negative L\'evy processes.
Moreover, we state the solution of the Shepp--Shiryaev optimal stopping
problem (\ref{problem2}) which will play an important role throughout
this paper.

\subsection{Spectrally negative L\'evy processes}
It is well known that a spectrally negative L\'evy process $X$ is
characterised by its L\'evy triplet $(\gamma,\sigma,\Pi)$, where
\mbox{$\sigma\geq0, \gamma\in\R$} and $\Pi$ is a measure on $(-\infty,0)$
satisfying the condition $\int_{(-\infty,0)}(1\wedge x^2) \Pi
(dx)<\infty$. By the L\'evy--It\^o decomposition, $X$ may be
represented in the form
%
\begin{equation}\label{LevyItodecomposition1}
X_t=\sigma B_t-\gamma t+X^{(1)}_t+X^{(2)}_t,
\end{equation}
where $\{B_t\dvtx t\geq0\}$ is a standard Brownian motion, $\{
X^{(1)}_t\dvtx t\geq0\}$ is a compound Poisson process with discontinuities
of magnitude bigger than or equal to one and $\{X_t^{(2)}\dvtx t\geq0\}$ is
a square integrable martingale with discontinuities of magnitude
strictly smaller than one and the three processes are mutually
independent. In particular, if $X$ is of bounded variation, the
decomposition reduces to
%
\begin{equation}\label{LevyItodecomposition2}
X_t=\mathtt{d}t-\eta_t,
\end{equation}
where $\mathtt{d}>0$, and $\{\eta_t\dvtx t\geq0\}$ is a driftless
subordinator. Furthermore, the spectral negativity of $X$ ensures
existence of the Laplace exponent $\psi$ of $X$, that is, $\E
[e^{\theta X_1}]=e^{\psi(\theta)}$ for $\theta\geq0$, which is
known to take the form
\renewcommand{\theequation}{\mbox{$\ast$}}
\begin{equation}
\label{Lexponent}
\psi(\theta)=-\gamma\theta+\frac{1}{2}
\sigma^2\theta^2+\int_{(-\infty,0)}
\bigl(e^{\theta x}-1-\theta x1_{\{x>-1\}} \bigr) \Pi(dx).
\end{equation}
Its right-inverse is defined by
\[
\Phi(q):=\sup\bigl\{\lambda\geq0\dvtx \psi(\lambda)=q\bigr\}
\]
for $q\geq0$.

For any spectrally negative L\'evy process having $X_0=0$ we
introduce the family of martingales
\[
\exp\bigl(cX_t-\psi(c)t\bigr),
\]
defined for any $c\in\R$ for which $\psi(c)=\log\E[\exp
(cX_1)]<\infty$, and further the corresponding family of measures $\{
\P^c\}$ with Radon--Nikodym derivatives
%
\setcounter{equation}{8}
\renewcommand{\theequation}{\arabic{equation}}
\begin{equation}\label{changeofmeasure}
\frac{d\P^c}{d\P} \bigg\vert_{\mathcal{F}_t}=\exp\bigl(cX_t-\psi(c)t
\bigr).
\end{equation}
For all such $c$ the measure $\P^c_x$ will denote the translation of
$\P^c$ under which $X_0=x$. In particular, under $\P_x^c$ the process
$X$ is still a spectrally negative L\'evy process; cf. Theorem 3.9
in \cite{kyprianou}.

\subsection{Scale functions}
A special family of functions associated with spectrally negative
L\'evy processes is that of scale functions (cf. \cite{kyprianou})
which are defined as follows. For $q\geq0$, the $q$-scale function
$W^{(q)}\dvtx \R\longrightarrow[0,\infty)$ is the unique function whose
restriction to $(0,\infty)$ is continuous and has Laplace transform
\[
\int_0^\infty e^{-\theta x}W^{(q)}(x)
\,dx=\frac{1}{\psi(\theta
)-q},\qquad \theta>\Phi(q),
\]
and is defined to be identically zero for $x\leq0$. Equally important
is the scale function $Z^{(q)}\dvtx \R\longrightarrow[1,\infty)$ defined by
\[
Z^{(q)}(x)=1+q\int_0^xW^{(q)}(z)
\,dz.
\]
The passage times of $X$ below and above $k\in\R$ are denoted by
\[
\tau_k^-=\inf\{t>0\dvtx X_t\leq k\} \quad\mbox{and}\quad
\tau_k^+=\inf\{ t>0\dvtx X_t\geq k\}.
\]
We will make use of the following four identities. For $q\geq0$ and
$x\in(a,b)$ it holds that
%
\begin{eqnarray}
\label{scale1}
\E_x \bigl[e^{-q\tau^+_b}I_{\{\tau^+_b<\tau^-_a\}} \bigr]&=&
\frac
{W^{(q)}(x-a)}{W^{(q)}(b-a)},
\\
\label{scale2}
\E_x \bigl[e^{-q\tau^-_a}I_{\{\tau^+_b>\tau^-_a\}} \bigr]&=&
Z^{(q)}(x-a)-W^{(q)}(x-a)
\frac{Z^{(q)}(b-a)}{W^{(q)}(b-a)}
\end{eqnarray}
for $q>0$ and $x\in\R$ it holds that
%
\begin{equation}\label{scale3}
\E_x \bigl[e^{-q\tau_0^-}1_{\{\tau_0^-<\infty\}} \bigr]=Z^{(q)}(x)-
\frac{q}{\Phi(q)}W^{(q)}(x);
\end{equation}
and finally for $q>0$ we have
%
\begin{equation}\label{scale4}
\lim_{x\to\infty}\frac{Z^{(q)}(x)}{W^{(q)}(x)}=\frac{q}{\Phi
(q)}.
\end{equation}
Identities (\ref{scale1}) and (\ref{scale2}) are Proposition 1
in \cite{exitproblems}, identity (\ref{scale4}) is Lemma 1 of \cite
{exitproblems} and (\ref{scale3}) can be found in Theorem 8.1 in \cite
{kyprianou}.
For each $c\geq0$ we denote by $W_c^{(q)}$ the $q$-scale function with
respect to the measure $\P^c$. A useful formula (cf. \cite
{kyprianou}) linking the scale function under different measures is
given by
%
\begin{equation}\label{scale5}
W^{(q)}(x)=e^{\Phi(q)x}W_{\Phi(q)}(x)
\end{equation}
for $q\geq0$ and $x\geq0$.

We conclude this subsection by stating some known regularity properties
of scale functions; cf. Lemma 2.4, Corollary 2.5, Theorem 3.10, Lemmas
3.1 and 3.2 of \cite{KuzKypRiv}.

\textit{Smoothness}: For all $q\geq0$,
\[
W^{(q)}\vert_{(0,\infty)}\in\cases{
C^1(0,\infty),& if $X$ is of bounded variation and $\Pi$ has no
atoms,
\cr
C^1(0,\infty),& if $X$ is of unbounded variation and $
\sigma=0$,
\cr
C^2(0,\infty),& $\sigma>0$.}
\]

\textit{Continuity at the origin}: For all $q\geq0$,
%
\begin{equation}
\label{continuityatorigin} W^{(q)}(0+)=\cases{\mathtt{d}^{-1},&
if $X$ is of bounded variation,
\cr
0, & if $X$ is of unbounded variation.}
\end{equation}

\textit{Derivative at the origin}: For all $q\geq0$,
%
\begin{equation}
\label{derivativeatorigin} W_+^{(q)\prime}(0+)=\cases{ \displaystyle \frac{q+\Pi
(-\infty,0)}{\mathtt{d}^2}, & if
$\sigma=0$ and $\Pi(-\infty,0)<\infty$,
\cr
\displaystyle \frac{2}{\sigma^2}, & if $\sigma>0$ or
$\Pi(-\infty,0)=\infty$,}
\end{equation}
where we understand the second case to be $+\infty$ when $\sigma=0$.

For technical reasons, we require for the rest of the paper
that $W^{(q)}$ is in $C^1(0,\infty)$ [and hence $Z^{(q)}\in
C^2(0,\infty)$]. This is ensured by henceforth assuming that $\Pi$ is
atomless whenever $X$ has paths of bounded variation.

\subsection{Solution to the Shepp--Shiryaev optimal stopping
problem}\label{ss}
In order to state the solution of the Shepp--Shiryaev optimal stopping
problem, we introduce the function $f\dvtx [0,\infty)\rightarrow\R$ which
is defined as
\[
f(z)=Z^{(q)}(z)-qW^{(q)}(z).
\]
It can be shown (cf. page 6 of \cite{baurdoux}) that, when $q>\psi
(1)$, the function $f$ is strictly decreasing to $-\infty$ and hence
within this regime
\[
k^*:=\inf\bigl\{z\geq0\dvtx Z^{(q)}(z)\leq qW^{(q)}(z)\bigr\}\in[0,
\infty).
\]
In particular, when $q>\psi(1)$, then $k^*=0$ if and only if
$W^{(q)}(0+)\geq q^{-1}$. Also, note that the requirement
$W^{(q)}(0+)\geq
q^{-1}$ implies $q\geq\mathtt{d}>\psi(1)$. We now give a
reformulation of a part of Theorem 1 in \cite{baurdoux}.
%
\begin{prop}\label{solution2}
\textup{(a)}
Suppose that $q>\psi(1)$ and $W^{(q)}(0+)<q^{-1}$. Then the
solution of (\ref{problem2}) is given by
\[
V^*(x,s)=e^sZ^{(q)}\bigl(x-s+k^*\bigr)
\]
with optimal strategy
\[
\tau^*:=\inf\bigl\{t\geq0\dvtx\overline X_t-X_t\geq k^*
\bigr\}.
\]

\textup{(b)} If $W^{(q)}(0+)\geq q^{-1}$ [and hence $q>\psi
(1)$], then the solution of (\ref{problem2}) is given by
$V^*(x,s)=e^s$ and optimal strategy $\tau^*=0$.

\textup{(c)} If $q\leq\psi(1)$, then $V^*(x,s)=\infty$.
\end{prop}
The result in part (b) of Proposition \ref
{solution2} is not surprising. If $W^{(q)}(0+)\geq q^{-1}$, then
$X$ is necessarily of bounded variation with $\mathtt{d}\leq q$ which
implies that the process $(e^{-qt+\overline X_t})_{t\geq0}$ is
pathwise decreasing. As a result we have for $\tau\in\mathcal
{M}$ the inequality $\E_{x,s} [e^{-q\tau+\overline X_\tau}
]\leq e^s$ and hence (b) follows. An analogous argument
shows that $V^*_\epsilon(x,s)=e^{s\wedge\epsilon}$ for $(x,s)\in E$
with optimal strategy $\tau^*_\epsilon=0$ and $V^*_{\epsilon
_1,\epsilon_2}(x,s)=e^{s\wedge\epsilon_2}$ for $(x,s)\in E$ with
optimal strategy $\tau^*_{\epsilon_1,\epsilon_2}=0$. Therefore, we
will not consider the regime $W^{(q)}(0+)\geq q^{-1}$ in what follows.
Note, however, that the parameter regime $q\leq\psi(1)$ will not be
degenerate for (\ref{problem1}) and (\ref{problem3}) due to the upper
cap which prevents the value function from exploding.

\section{Main results}\label{mainre}
\subsection{Maximum process with upper cap}
The first result ensures existence of a function $g_\epsilon$ which,
as will follow in due course, describes the optimal stopping boundary
in (\ref{problem1}).
%
\begin{lem}\label{g}
Let $\epsilon\in\R$ be given.
\begin{longlist}[(b)]
\item[(a)]
If $q>\psi(1)$ and $W^{(q)}(0+)<q^{-1}$, then
$k^*\in(0,\infty)$.
\item[(b)] If $q\leq\psi(1)$, then $k^*=\infty$.
\item[(c)] Under the assumptions in \textup{(a)} or \textup{(b)}, there
exists a unique solution $g_\epsilon\dvtx (-\infty,\epsilon)\to
(0,k^*)$ of the ordinary differential equation
%
\begin{equation}\label{diffequ}
g_\epsilon'(s)=1-\frac{Z^{(q)}(g_\epsilon(s))}{qW^{(q)}(g_\epsilon
(s))} \qquad\mbox{on $(-\infty,
\epsilon)$}
\end{equation}
satisfying $\lim_{s\uparrow\epsilon}g_\epsilon(s)=0$ and
$\lim_{s\to-\infty}g_\epsilon(s)=k^*$.
\end{longlist}
\end{lem}
Next, extend $g_\epsilon$ to the whole real line by setting
$g_\epsilon(s)=0$ for $s\geq\epsilon$. We now present the solution
of (\ref{problem1}).
%
\begin{theorem}\label{rwcmainresult}
Let $\epsilon\in\R$ be given and suppose that $q>\psi(1)$
and\break
$W^{(q)}(0+)<q^{-1}$ or $q\leq\psi(1)$. Then the solution of (\ref
{problem1}) is given by
\[
V^*_\epsilon(x,s)=e^{s\wedge\epsilon}Z^{(q)}\bigl(x-s+g_\epsilon(s)
\bigr)
\]
with corresponding optimal strategy
\[
\tau_\epsilon^*:=\inf\bigl\{t\geq0\dvtx\overline X_t-X_t
\geq g_\epsilon(\overline X_t)\bigr\},
\]
where $g_\epsilon$ is given in Lemma \ref{g}.
\end{theorem}

Define the continuation region
\[
C_\epsilon^*=C^*:=\bigl\{(x,s)\in E \vert s<\epsilon,s-g_\epsilon
(s)<x\leq s\bigr\}
\]
and the stopping region $D^*_\epsilon=D^*:=E\setminus C^*$. The shape
of the boundary separating them, that is, the optimal stopping
boundary, is of particular interest. Theorem \ref{rwcmainresult}
together with (\ref{continuityatorigin}) and (\ref{diffequ}) shows that
\[
\lim_{s\uparrow\epsilon}g_\epsilon'(s)=\cases{-\infty,
&\quad if $X$ is of unbounded variation,
\cr
1-\mathtt{d}/q, &\quad if $X$ is of bounded
variation.}
\]
Also, using (\ref{scale4}) we see that
\[
\lim_{s\to-\infty}g_\epsilon^\prime(s)=\cases{0, &\quad if
$q>\psi(1)$ and $W^{(q)}(0+)<q^{-1}$,
\cr
1-
\Phi(q)^{-1}, &\quad if $q\leq\psi(1)$.}
\]
This (qualitative) behaviour of $g_\epsilon$ and the resulting shape of
the continuation and stopping region are illustrated in Figure \ref
{figbehaviour}. Note in particular that the shape of $g_\epsilon$
at~$\epsilon$ (and consequently the optimal boundary) changes according
to the path variation of $X$.
%
\begin{figure}

\includegraphics{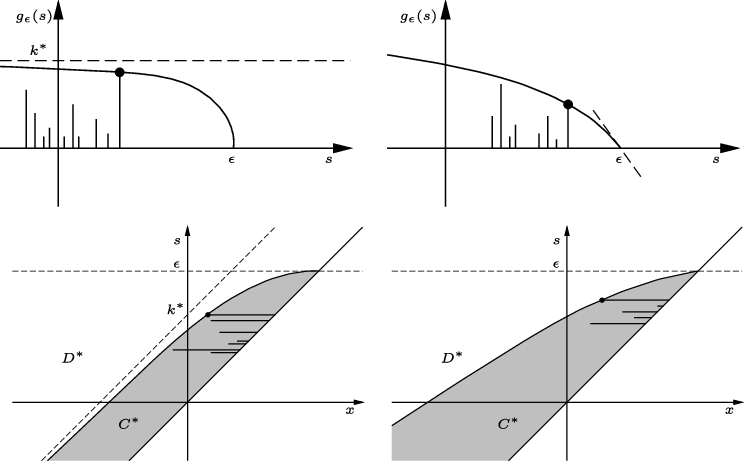}

\caption{For the two pictures on the left it is assumed that $q>\psi
(1)$ and $W^{(q)}(0+)=0$ whereas on the right it is assumed that $q\leq
\psi(1)$.}\label{figbehaviour}
\end{figure}
The horizontal and vertical lines in Figure \ref{figbehaviour} are
meant to schematically indicate the trace of the excursions of $X$ away
from the running maximum. We thus see that the optimal strategy
consists of continuing if the height of the excursion away from the
running supremum $s$ does not exceed $g_\epsilon(s)$; otherwise we stop.

\subsection{Maximum process with upper and lower cap}\label{Mpwualc}
Inspired by the result in~\cite{abarrierversion}, we expect the
strategy $T_{\epsilon_1}\wedge\tau_{\epsilon_2}^*$ to be optimal, where
$\tau_{\epsilon_2}^*$ is given in Theorem \ref{rwcmainresult} and
$T_{\epsilon_1}=\inf\{t\geq0\dvtx X_t\leq\epsilon_1\}$. This means that the
optimal boundary is expected to be a vertical line at $\epsilon_1$
combined with the curve described by $g_{\epsilon_2}$ characterised in
Lemma \ref{g}. Before we can proceed, we need to introduce an auxiliary
quantity, namely the point on the $s$-axis where the vertical line at
$\epsilon_1$ and the optimal boundary corresponding to $g_{\epsilon_2}$
intersect; see Figure \ref{figcombined}. If $q>\psi(1)$ and
$W^{(q)}(0+)<q^{-1}$ or $q\leq\psi(1)$ define the map
$a_{\epsilon_2}\dvtx (-\infty,\epsilon_2)\to(0,k^*)$ by
%
\begin{figure}

\includegraphics{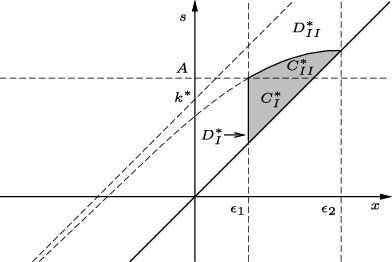}

\caption{A qualitative picture of the continuation and stopping region
under the assumption that $q>\psi(1)$ and $W^{(q)}(0+)=0$; cf.
Theorem \protect\ref{combinationmainresult}.}\label{figcombined}
\end{figure}
$a_{\epsilon_2}(s):=s-g_{\epsilon_2}(s)$. It follows by definition of
$g_{\epsilon_2}$ that $a_{\epsilon_2}$ is continuous, strictly
increasing and satisfies $\lim_{s\uparrow\epsilon_2}a_{\epsilon
_2}(s)=\epsilon_2$ and $\lim_{s\downarrow-\infty}a_{\epsilon
_2}(s)=-\infty$. Therefore the intermediate value theorem
guarantees existence of a unique
$A_{\epsilon_1,\epsilon_2}=A\in(-\infty,\epsilon_2)$ such that
$A-g_{\epsilon_2}(A)=\epsilon_1$. Our candidate optimal strategy
$T_{\epsilon_1}\wedge\tau_{\epsilon_2}^*$ splits $E_{\epsilon_1} :=
\{(x, s) \in E\dvtx x \geq \epsilon_1\}$ into the stopping regions
\begin{eqnarray*}
D_{I,\epsilon_1,\epsilon_2}^*&=&D_I^*:=\bigl\{(x,s)\in E\dvtx x=\epsilon
_1,\epsilon_1\leq s\leq A\bigr\},
\\
D_{\mathit{II},\epsilon_1,\epsilon_2}^*&=&D_{\mathit{II}}^*:=\bigl\{(x,s)\in E\dvtx
\epsilon
_1\leq x\leq s-g_{\epsilon_2}(s),s>A\bigr\}
\end{eqnarray*}
and the continuation regions
\begin{eqnarray*}
C_{I,\epsilon_1,\epsilon_2}^*&=&C_I^*:=\bigl\{(x,s)\in E\dvtx\epsilon
_1<x\leq s,\epsilon_1<s<A\bigr\},
\\
C_{\mathit{II},\epsilon_1,\epsilon_2}^*&=&C_{\mathit{II}}^*:=\bigl\{(x,s)\in E\dvtx
s-g_{\epsilon_2}(s)<x\leq s,A\leq s<\epsilon_2\bigr\}.
\end{eqnarray*}
%
Clearly, if $(x,s)\in E\setminus E_{\epsilon_1}$, then the only
stopping time in $\mathcal{M}_{\epsilon_1}$ is $\tau=0$ and hence the
optimal value function is given by $e^{s\wedge\epsilon_2}$.
Furthermore, when \mbox{$(x,s)\in C^*_{\mathit{II}}\cup
D^*_{\mathit{II}}$} we have $\tau _{\epsilon_2}^*\leq T_{\epsilon_1}$,
so that the optimality\vspace*{1pt} of $\tau _{\epsilon_2}^*$ in (\ref{problem1})
implies \mbox{$V^*_{\epsilon
_1,\epsilon_2}(x,s)=V_{\epsilon_2}^*(x,s)$}. Consequently,\vspace*{1pt} the
interesting case is really \mbox{$(x,s)\in C^*_I\cup D^*_I$}. The
key to verifying that $T_{\epsilon_1}\wedge\tau_{\epsilon_2}^*$ is
optimal, is to find the value function associated with it.
%
\begin{lem}\label{valuefunction1}
Let $\epsilon_1<\epsilon_2$ be given, and suppose that
$q>\psi(1)$
and\break $W^{(q)}(0+)<q^{-1}$ or $q\leq\psi(1)$. Define
\[
V_{\epsilon_1,\epsilon_2}(x,s):=\cases{V_{\epsilon_2}^*(x,s), &\quad
$(x,s)\in
C_{\mathit{II}}^*\cup D_{\mathit{II}}^*$,
\vspace*{1pt}\cr
U_{\epsilon_1,\epsilon_2}(x,s), &\quad $(x,s)
\in C_I^*\cup D_I^*$,
\vspace*{1pt}\cr
e^{s\wedge\epsilon_2}, &\quad
otherwise,}
\]
where $V_{\epsilon_2}^*$ is given in Theorem \ref{rwcmainresult},
\[
U_{\epsilon_1,\epsilon_2}(x,s):=e^sZ^{(q)}(x-\epsilon_1)+e^{\epsilon
_1}W^{(q)}(x-
\epsilon_1)\int_{s-\epsilon_1}^{g_{\epsilon
_2}(A)}e^t
\frac{Z^{(q)}(t)}{W^{(q)}(t)}\,dt
\]
and $A\in(-\infty,\epsilon_2)$ is the unique constant such that
$A-g_{\epsilon_2}(A)=\epsilon_1$. We then have, for $(x,s)\in E$,
\[
\E_{x,s} \bigl[e^{-q(T_{\epsilon_1}\wedge\tau^*_{\epsilon
_2})+\overline X_{T_{\epsilon_1}\wedge\tau^*_{\epsilon_2}}\wedge
\epsilon_2} \bigr]=V_{\epsilon_1,\epsilon_2}(x,s).
\]
\end{lem}

Our main contribution here is the expression for $U_{\epsilon
_1,\epsilon_2}$, thereby allowing us to verify that the strategy
$T_{\epsilon_1}\wedge\tau_{\epsilon_2}^*$ is still optimal. In
fact, this is the content the following result.
%
\begin{theorem}\label{combinationmainresult}
Let $\epsilon_1<\epsilon_2$ be given and suppose that $q>\psi(1)$
and $W^{(q)}(0+)<q^{-1}$ or $q\leq\psi(1)$. Then the solution
to (\ref
{problem3}) is given by $V_{\epsilon_1,\epsilon_2}^*=V_{\epsilon
_1,\epsilon_2}$ with corresponding optimal strategy $\tau_{\epsilon
_1,\epsilon_2}^*=T_{\epsilon_1}\wedge\tau^*_{\epsilon_2}$, where
$\tau^*_{\epsilon_2}$ is given in Theorem \ref{rwcmainresult}.
\end{theorem}
It is also possible to obtain the solution of (\ref{problem3}) with
lower cap only. To this end, define when $q>\psi(1)$ and
$W^{(q)}(0+)<q^{-1}$ the constant function $g_\infty(s):=k^*$ and
$A_{\epsilon_1,\infty}:=\epsilon_1+k^*$.
%
\begin{cor}\label{consistencybarrier}
Let $\epsilon_1\in\R$ and suppose that $\epsilon_2=\infty$, that
is, there is no upper cap.
\begin{longlist}[(b)]
\item[(a)] Assume that $q>\psi(1)$ and that
$W^{(q)}(0+)<q^{-1}$. Then the solution to (\ref{problem3}) is given by
%
\begin{equation}\label{limitp}
V^*_{\epsilon_1,\infty}(x,s)=\cases{V^*(x,s), &\quad $(x,s)\in C_{\mathit{II},\epsilon
_1,\infty}^*\cup
D_{\mathit{II},\epsilon_1,\infty}^*$,
\vspace*{1pt}\cr
U_{\epsilon_1,\infty}(x,s), &\quad $(x,s)\in C_{I,\epsilon_1,\infty
}^*
\cup D_{I,\epsilon_1,\infty}^*$,
\vspace*{1pt}\cr
e^s, &\quad otherwise,}
\end{equation}
where $V^*$ is given in Proposition \ref{solution2} and
\[
U_{\epsilon_1,\infty}(x,s)=e^sZ^{(q)}(x-\epsilon_1)+e^{\epsilon
_1}W^{(q)}(x-
\epsilon_1)\int_{s-\epsilon_1}^{k^*}e^t
\frac
{Z^{(q)}(t)}{W^{(q)}(t)}\,dt.
\]
The corresponding optimal strategy is given by $\tau^*_{\epsilon
_1,\infty}=T_{\epsilon_1}\wedge\tau^*$, where $\tau^*$ is given in
Proposition \ref{solution2}.
\item[(b)] If $q\leq\psi(1)$, then $V^*_{\epsilon
_1,\infty}(x,s)=\infty$ for $(x,s)\in E_{\epsilon_1}$ and
$V^*_{\epsilon_1,\infty}(x,s)=e^s$ otherwise.
\end{longlist}
\end{cor}
%
\begin{rem}
In Theorem \ref{rwcmainresult} there is no lower cap, and hence it
seems natural to obtain Theorem \ref{rwcmainresult} as a corollary to
Theorem \ref{combinationmainresult}. This would be possible if one
merged the proofs of Theorems \ref{rwcmainresult} and \ref
{combinationmainresult} appropriately. However, a merged proof would
still contain the main arguments of both the proof of Theorem \ref
{rwcmainresult} and the proof of Theorem \ref{combinationmainresult}
(note that the proof of Theorem \ref{combinationmainresult} makes use
of Theorem \ref{rwcmainresult}). Therefore, and also for presentation
purposes, we chose to present them separately.
\end{rem}
Finally, if $X_t=(\mu-\frac{1}{2}\sigma^2)t+\sigma W_t$, where $\mu
\in\R,\sigma>0$ and $(W_t)_{t\geq0}$ is a standard Brownian motion,
then Corollary \ref{consistencybarrier} is nothing else than Theorem
3.1 in~\cite{abarrierversion}. However, this is not immediately
clear and requires a simple but lengthy computation which is provided
in Section \ref{applications}.

\section{Guess and verify via principle of smooth or continuous
fit}\label{firsttheorem}
Let us consider the solution to (\ref{problem1}) from an intuitive
point of view. We shall restrict ourselves to the case where $q>\psi
(1)$ and $W^{(q)}(0+)<q^{-1}$. It follows from what was said at the
beginning of Section \ref{ss} that $k^*\in(0,\infty)$.

It is clear that if $(x,s)\in E$ such that $x\geq\epsilon$, then it
is optimal to stop immediately since one cannot obtain a higher payoff
than $\epsilon$, and waiting is penalised by exponential discounting.
If $x$ is much smaller than $\epsilon$, then the cap $\epsilon$
should not have too much influence, and one expects that the optimal
value function $V^*_\epsilon$ and the corresponding optimal strategy
$\tau^*_\epsilon$ look similar to the optimal value function\vadjust{\goodbreak} $V^*$
and optimal strategy $\tau^*$ of problem (\ref{problem2}). On the
other hand, if $x$ is close to the cap, then the process $X$ should be
stopped ``before'' it is a distance $k^*$ away from its running maximum.
This can be explained as follows: the constant $k^*$ in the solution to
problem (\ref{problem2}) quantifies the acceptable ``waiting time'' for
a possibly much higher running supremum at a later point in time. But
if we impose a cap, there is no hope for a much higher supremum and
therefore ``waiting the acceptable time'' for problem (\ref{problem2})
does not pay off in the situation with cap. With exponential
discounting we would therefore expect to exercise earlier. In other
words, we expect an optimal strategy of the form
\[
\tau_{g_\epsilon}=\inf\bigl\{t\geq0\dvtx\overline X_t-X_t
\geq g_\epsilon(\overline X)\bigr\}
\]
for some function $g_\epsilon$ satisfying $\lim_{s\to-\infty
}g_\epsilon(s)=k^*$ and $\lim_{s\to\epsilon}g_\epsilon(s)=0$.

This qualitative guess can be turned into a quantitative guess by an
adaptation of the argument in Section 3 of \cite{maximalityprinciple}
to our setting. To this end, assume that $X$ is of unbounded variation
($W^{(q)}(0+)=0$). We will deal with the bounded variation case later.
From the general theory of optimal stopping (cf. \cite{peskir},
Section 13) we informally expect the value function
\[
V_{g_\epsilon}(x,s)=\E_{x,s} \bigl[e^{-q\tau_{g_\epsilon}+\overline
X_{\tau_{g_\epsilon}}} \bigr]
\]
to satisfy the system
%
\begin{eqnarray}\label{system}
\Gamma V_{g_\epsilon}(x,s)&=&qV_{g_\epsilon}(x,s)\qquad\mbox{for
$s-g_{\epsilon}(s)<x<s$ with $s$ fixed},
\nonumber
\\
\frac{\partial V_{g_\epsilon}}{\partial s}(x,s) \bigg\vert_{x=s-}&=&0
\qquad\mbox{(normal reflection),}
\\
\qquad V_{g_\epsilon}(x,s)\vert_{x=(s-g_\epsilon(s))+}&=&e^s
\qquad\mbox{(instantaneous stopping),}
\nonumber
\end{eqnarray}
where $\Gamma$ is the infinitesimal generator of the process $X$ under
$\P_0$. Moreover, the principle of smooth fit \cite
{peskir,mikhalevich} suggests that this system should be complemented by
%
\begin{equation}\label{sfit}
\frac{\partial V_{g_\epsilon}}{\partial x}(x,s) \bigg\vert_{x=(s-g_\epsilon
(s))+}=0 \qquad\mbox{(smooth fit).}
\end{equation}
Note that, although the smooth fit condition is not necessarily part of
the general theory, it is imposed since by the ``rule of thumb''
outlined in Section 7 in \cite{someremarks} it should hold in this
setting because of path regularity. This belief will be vindicated when
we show that system (\ref{system}) with (\ref{sfit}) leads to the
solution of problem (\ref{problem1}).
Applying the strong Markov property at $\tau^+_s$ and using (\ref
{scale1}) and (\ref{scale2}) shows that
\begin{eqnarray*}
V_{g_\epsilon}(x,s)&=&e^s\E_{x,s} \bigl[e^{-q\tau_{s-g_\epsilon
(s)}^-}1_{\{\tau^-_{s-g_\epsilon(s)}<\tau^+_s\}}
\bigr]
\\
&&{}+\E_{x,s} \bigl[e^{-q\tau^+_s}1_{\{\tau^-_{s-g_\epsilon(s)}>\tau
^+_s\}} \bigr]
\E_{s,s} \bigl[e^{-q\tau_{g_\epsilon}+\overline X_{\tau
_{g_\epsilon}}} \bigr]
\\
&=&e^s \biggl(Z^{(q)}\bigl(x-s+g_\epsilon(s)
\bigr)-W^{(q)}\bigl(x-s+g_\epsilon(s)\bigr)\frac
{Z^{(q)}(g_\epsilon(s))}{W^{(q)}(g_\epsilon(s))}
\biggr)
\\
&&{}+\frac{W^{(q)}(x-s+g_\epsilon(s))}{W^{(q)}(g_\epsilon
(s))}V_{g_\epsilon}(s,s).
\end{eqnarray*}
Furthermore, the smooth fit condition implies
\begin{eqnarray*}
0&=&\lim_{x\downarrow s-g_\epsilon(s)}\frac{\partial V_{g_\epsilon
}}{\partial x}(x,s)
\\
&=&\lim_{x\downarrow s-g_\epsilon(s)}\frac{W^{(q)\prime
}(x-s+g_\epsilon(s))}{W^{(q)}(g_\epsilon(s))} \bigl(V_{g_\epsilon
}(s,s)-e^sZ^{(q)}
\bigl(g_\epsilon(s)\bigr) \bigr).
\end{eqnarray*}
By (\ref{derivativeatorigin}) the first factor tends to a strictly
positive value or infinity which shows that $V_{g_\epsilon
}(s,s)=e^sZ^{(q)}(g_\epsilon(s))$. This would mean that for $(x,s)\in
E$ such that
\mbox{$s-g_\epsilon(s)<x<s$} we have
%
\begin{equation}\label{candidateV}
V_{g_\epsilon}(x,s)=e^sZ^{(q)}\bigl(x-s+g_\epsilon(s)
\bigr).
\end{equation}
Having derived the form of a candidate optimal value function
$V_{g_\epsilon}$, we still need to do the same for $g_\epsilon$.
Using the normal reflection condition in (\ref{system}) shows that our
candidate function $g_\epsilon$ should satisfy the ordinary
differential equation
\[
Z^{(q)}\bigl(g_\epsilon(s)\bigr)+qW^{(q)}
\bigl(g_\epsilon(s)\bigr) \bigl(g^\prime_\epsilon(s)-1
\bigr)=0.
\]

If $X$ is of bounded variation ($W^{(q)}(0+)\in(0,q^{-1})$), we
informally expect from the general theory that $V_{g_\epsilon}$
satisfies the first two equations of (\ref{system}). Additionally, the
principle of continuous fit \cite{someremarks,pesshir} suggests that
the system should be complemented by
\[
V_{g_\epsilon}(x,s)\vert_{x=(s-g_\epsilon(s))+}=e^s \qquad\mbox{(continuous
fit).}
\]
A very similar argument as above produces the same candidate value
function and the same ordinary differential equation for $g_\epsilon$.
\section{Proofs of main results}\label{mainrepf}
\mbox{}
\begin{pf*}{Proof of Lemma \ref{g}}
The idea is to define a suitable bijection $H$ from $(0,k^*)$ to
$(-\infty,\epsilon)$ whose inverse satisfies the differential
equation and the boundary conditions.

First consider the case $q>\psi(1)$ and $W^{(q)}(0+)<q^{-1}$. It
follows from the discussion at the beginning of Section \ref{ss}
that $k^*\in(0,\infty)$ and that the function \mbox{$s\mapsto h(s):=1-\frac
{Z^{(q)}(s)}{qW^{(q)}(s)}$} is negative on $(0,k^*)$. Moreover, $\lim
_{s\downarrow0}h(s)\in[-\infty,0)$ and $\lim_{s\uparrow
k^*}h(s)=0$. These properties\vadjust{\goodbreak} imply that the function
$H\dvtx\break (0,k^*)\rightarrow(-\infty,\epsilon)$ defined by
%
\begin{eqnarray}\label{defofH}
H(s):\!&=&\epsilon+\int_0^s \biggl(1-
\frac{Z^{(q)}(\eta
)}{qW^{(q)}(\eta)} \biggr)^{-1}\,d\eta\nonumber\\[-8pt]\\[-8pt]
&=&\epsilon+\int
_0^s\frac
{qW^{(q)}(\eta)}{qW^{(q)}(\eta)-Z^{(q)}(\eta)}\,d\eta\nonumber
\end{eqnarray}
is strictly decreasing. If we can also show that the integral tends to
$-\infty$ as $s$ approaches $k^*$, we could deduce that $H$ is a
bijection from $(0,k^*)$ to $(-\infty,\epsilon)$. Indeed, appealing
to l'H\^opital's rule and using (\ref{scale3}) we obtain
\begin{eqnarray*}
\lim_{z\uparrow k^*}\frac{qW^{(q)}(z)-Z^{(q)}(z)}{k^*-z}&=&\lim
_{z\uparrow k^*}qW^{(q)}(z)-qW^{(q)\prime}(z)
\\
&=&\lim_{z\uparrow k^*}qe^{\Phi(q)z} \bigl(\bigl(1-\Phi(q)
\bigr)W_{\Phi
(q)}(z)-W_{\Phi(q)}'(z) \bigr)
\\
&=&qe^{\Phi(q)k^*} \bigl(\bigl(1-\Phi(q)\bigr)W_{\Phi(q)}\bigl(k^*
\bigr)-W_{\Phi
(q)}'\bigl(k^*\bigr) \bigr).
\end{eqnarray*}
Denote the term on the right-hand side by $c$, and note that $c<0$ due
to the fact that $W_{\Phi(q)}$ is strictly positive and increasing on
$(0,\infty)$ and since $\Phi(q)>1$ for \mbox{$q>\psi(1)$}. Hence there
exists a $\delta>0$ and $0<z_0<k^*$ such that \mbox{$c-\delta<\frac
{qW^{(q)}(z)-Z^{(q)}(z)}{k^*-z}$} for all $z_0<z<k^*$. Thus
\[
\frac{1}{qW^{(q)}(z)-Z^{(q)}(z)}<\frac{1}{(c-\delta)(k^*-z)}<0
\qquad\mbox{for $z_0<z<k^*$.}
\]
This shows that
\[
\lim_{s\uparrow k^*}H(s)\leq\epsilon+\lim_{s\uparrow k^*}\int
_{z_0}^{s}\frac{qW^{(q)}(\eta)}{(c-\delta)(k^*-\eta)}\,d\eta=-\infty.
\]

The discussion above permits us to define $g_\epsilon:=H^{-1}\in
C^1((-\infty,\epsilon);(0,k^*))$. In particular, differentiating
$g_\epsilon$ gives
\[
g^\prime_\epsilon(s)=\frac{1}{H^\prime(g_\epsilon(s))}=1-\frac
{Z^{(q)}(g_\epsilon(s))}{qW^{(q)}(g_\epsilon(s))}
\]
for $s\in(-\infty,\epsilon)$, and $g_\epsilon$ satisfies $\lim_{s\to
-\infty}g_\epsilon(s)=k^*$ and $\lim_{s\uparrow\epsilon
}g_\epsilon(s)=0$ by construction.

As for the case $q\leq\psi(1)$, note that by (\ref{scale3})
and (\ref{scale4}) we have
%
\begin{equation}\label{estimatefornegativity}
Z^{(q)}(x)-qW^{(q)}(x)\geq Z^{(q)}(x)-
\frac{q}{\Phi
(q)}W^{(q)}(x)>0
\end{equation}
for $x\geq0$ which shows that $k^*=\infty$. Moreover, (\ref
{estimatefornegativity}) together with (\ref{scale4}) implies that
the map $s\mapsto h(s)$ is negative on $(0,\infty)$ and satisfies
$\lim_{s\downarrow0} h(s)\in[-\infty,0)$ and $\lim_{s\uparrow
\infty} h(s)=1-\Phi(q)^{-1}\leq0$. Defining $H\dvtx (0,\infty
)\rightarrow(-\infty,\epsilon)$ as in (\ref{defofH}), one deduces
similarly as above that $H$ is a continuously differentiable bijection
whose inverse satisfies the requirements.

We finish the proof by addressing the question of uniqueness. To this
end, assume that there is another solution $\tilde g$. In particular,
$\tilde g^\prime(s)= h(\tilde g(s))$ for $s\in(s_1,\epsilon)\subset
(-\infty,\epsilon)$ and hence
\[
s_1=\epsilon-\int_{(s_1,\epsilon)}d\eta=\epsilon+\int
_{(s_1,\epsilon)}\frac{\llvert\tilde g^\prime(s)\rrvert}{h(\tilde
g(s))}\, ds=\epsilon+\int
_0^{\tilde g(s_1)}\frac{1}{h(s)}\,ds=H\bigl(\tilde
g(s_1)\bigr),
\]
which implies that $\tilde g=H^{-1}=g_\epsilon$.
\end{pf*}
\begin{pf*}{Proof of Theorem \ref{rwcmainresult}}
Define the function
\[
V_\epsilon(x,s):=e^{s\wedge\epsilon}Z^{(q)}\bigl(x-s+g_\epsilon(s)
\bigr)
\]
for $(x,s)\in E$, and let $\tau_{g_\epsilon}:=\inf\{t\geq
0\dvtx \overline X_t-X_t\geq g_\epsilon(\overline X_t)\}$, where
$g_\epsilon$ is as in Lemma \ref{g}. Because of the infinite horizon
and Markovian claim structure of problem (\ref{problem1}) it is enough
to check the following conditions:
\begin{longlist}[(iii)]
\item[(i)] $V_\epsilon(x,s)\geq e^{s\wedge\epsilon}$ for
all $(x,s)\in E$;
\item[(ii)] $\{e^{-qt}V_\epsilon(X_t,\overline X_t)\dvtx t\geq0\}$ is a
right-continuous $\P_{x,s}$-supermartingale for $(x,s)\in E$;
\item[(iii)] $V_\epsilon(x,s)=\E_{x,s} [e^{-q\tau
_{g_\epsilon}+\overline X_{\tau_{g_\epsilon}}\wedge\epsilon}
]$ for all $(x,s)\in E$.
\end{longlist}
To see why these are sufficient conditions, note that (i)
and (ii) together with Fatou's lemma in the second inequality
and Doob's stopping theorem in the third inequality show that for $\tau
\in\mathcal{M}$,
\begin{eqnarray*}
\E_{x,s} \bigl[e^{-q\tau+\overline X_\tau\wedge\epsilon} \bigr]&\leq&\E_{x,s}
\bigl[e^{-q\tau}V_\epsilon(X_\tau,\overline
X_\tau) \bigr]
\\
&\leq&\liminf_{t\to\infty}\E_{x,s} \bigl[e^{-q(t\wedge\tau
)}V_\epsilon(X_{t\wedge\tau},
\overline X_{t\wedge\tau}) \bigr]
\\
&\leq&V_\epsilon(x,s),
\end{eqnarray*}
which in view of (iii) implies $V^*_\epsilon=V_\epsilon$ and
$\tau^*_\epsilon=\tau_{g_\epsilon}$.

The remainder of this proof is devoted to checking conditions (i)--(iii).
Clearly, condition (i) is satisfied since $Z^{(q)}$ is
bigger or equal to one by definition.\vspace*{9pt}

\textit{Supermartingale property} (ii).\quad
Given the inequality
%
\begin{equation}\label{inequality1}
\E_{x,s} \bigl[e^{-qt}V_\epsilon(X_t,
\overline X_t) \bigr]\leq V_\epsilon(x,s),\qquad (x,s)\in E,
\end{equation}
the supermartingale property is a consequence of the Markov property of
the process $(X,\overline X)$. Indeed, for $u\leq t$ we have
\begin{eqnarray*}
\E_{x,s} \bigl[e^{-qt}V_\epsilon(X_t,
\overline X_t) \vert\mathcal{F}_u \bigr]&=&e^{-qu}
\E_{X_u,\overline X_u} \bigl[e^{-q(t-u)}V_\epsilon(X_{t-u},
\overline X_{t-u}) \bigr]
\\
&\leq& e^{-qu}V_\epsilon(X_u,\overline
X_u).
\end{eqnarray*}

We now prove (\ref{inequality1}), first under the assumption that
$W^{(q)}(0+)=0$, that is, $X$~is of unbounded variation. Let $\Gamma$ be
the infinitesimal generator of $X$ and formally define the function
$\Gamma Z^{(q)}\dvtx \R\setminus\{0\}\rightarrow\R$ by
\begin{eqnarray*}
\Gamma Z^{(q)}(x)&:=&-\gamma Z^{(q)\prime}(x)+\frac{\sigma
^2}{2}Z^{(q)\prime\prime}(x)
\\
&&{}+\int_{(-\infty,0)} \bigl(Z^{(q)}(x+y)-Z^{(q)}(x)-yZ^{(q)\prime
}(x)1_{\{
y\geq-1\}}
\bigr) \Pi(dy).
\end{eqnarray*}
For $x<0$ the quantity $\Gamma Z^{(q)}(x)$ is well defined and $\Gamma
Z^{(q)}(x)=0$. However, for $x>0$ one needs to check whether the
integral part of $\Gamma Z^{(q)}(x)$ is well defined. This is done in
Lemma \ref{generator1} in the \hyperref[app]{Appendix} which shows that this is
indeed the case. Moreover, as shown in Section 3.2 of \cite{pist}, it
holds that
\[
\Gamma Z^{(q)}(x)=qZ^{(q)}(x),\qquad x\in(0,\infty).
\]

Now fix $(x,s)\in E$ and define the semimartingale $Y_t:=X_t-\overline
X_t+g_\epsilon(\overline X_t)$. Applying an appropriate version of the
It\^o--Meyer formula (cf. Theorem 71, Chapter~IV of \cite{protter}) to
$Z^{(q)}(Y_t)$ yields $\P_{x,s}$-a.s.
%
\begin{eqnarray}\label{timeatzero1}
Z^{(q)}(Y_t)&=&Z^{(q)}\bigl(x-s+g_\epsilon(s)
\bigr)+m_t+\int_0^t\Gamma
Z^{(q)}(Y_u)\,du
\nonumber\\[-8pt]\\[-8pt]
&&{}+\int_0^tZ^{(q)\prime}(Y_u)
\bigl(g^\prime_\epsilon(\overline X_u)-1\bigr)\, d
\overline X_u,\nonumber
\end{eqnarray}
where
\begin{eqnarray*}
m_t&=&\int_{0+}^t\sigma
Z^{(q)\prime}(Y_{u-})\,dB_u+\int_{0+}^tZ^{(q)\prime}(Y_{u-})\,dX^{(2)}_u
\\
&&{}+\sum_{0<u\leq t} \bigl(\Delta Z^{(q)}(Y_u)-
\Delta X_uZ^{(q)\prime
}(Y_{u-})1_{\{\Delta X_u\geq-1\}} \bigr)
\\
&&{}-\int_0^t\int_{(-\infty,0)}
\bigl(Z^{(q)}(Y_{u-}+y)-Z^{(q)}(Y_{u-})\\
&&\qquad\hspace*{68.2pt}{}-yZ^{(q)\prime}(Y_{u-})1_{\{y\geq
-1\}}
\bigr) \Pi(dy)\,du
\end{eqnarray*}
and $\Delta X_u=X_u-X_{u-}, \Delta
Z^{(q)}(Y_u)=Z^{(q)}(Y_u)-Z^{(q)}(Y_{u-})$. The fact that $\Gamma
Z^{(q)}$ is not defined at zero is
not a problem as the time $Y$ spends at zero has Lebesgue measure zero
anyway. By the boundedness of $Z^{(q)\prime}$ on $(-\infty,g_\epsilon
(s)]$ the first two stochastic integrals in the expression for $m_t$
are zero-mean martingales, and by the compensation formula
(cf. Corollary 4.6 of \cite{kyprianou}) the third and fourth term
constitute a zero-mean martingale. Next, recall that \mbox{$V_\epsilon
(x,s)=e^{s\wedge\epsilon}Z^{(q)}(x-s+g_\epsilon(s))$} and use
stochastic integration by parts for semimartingales (cf. Corollary 2 of
Theorem 22, Chapter II of~\cite{protter}) to deduce that
%
\begin{eqnarray}\label{final}
&&
e^{-qt}V_\epsilon(X_t,\overline
X_t)\nonumber\\
&&\qquad=V_\epsilon(x,s)+M_t
\nonumber\\[-8pt]\\[-8pt]
&&\qquad\quad{}+\int_0^te^{-qu+\overline X_u\wedge\epsilon}\bigl(\Gamma
Z^{(q)}(Y_u)-qZ^{(q)}(Y_u)\bigr)\,du
\nonumber\\
&&\qquad\quad{}+\int_0^te^{-qu+\overline X_u\wedge\epsilon}
\bigl(Z^{(q)}(Y_u)1_{\{
\overline X_u\leq\epsilon\}}+Z^{(q)\prime}(Y_u)
\bigl(g^\prime_\epsilon(\overline X_u)-1\bigr)
\bigr)\,d\overline X_u,
\nonumber
\end{eqnarray}
where $M_t=\int_{0+}^te^{-qu+\overline X_u\wedge\epsilon}\,dm_u$ is a
zero-mean martingale. The first integral is nonpositive since $\Gamma
Z^{(q)}(y)-qZ^{(q)}(y)\leq0$ for all $y\in\R\setminus\{0\}$. The last integral
vanishes since the process $\overline X_u$ only increments when
$\overline X_u=X_u$ and by definition of $g_\epsilon$. Thus, taking
expectations on both sides yields
\[
\E_{x,s} \bigl[e^{-qt}V_\epsilon(X_t,
\overline X_t) \bigr]\leq V_\epsilon(x,s).
\]

If $W^{(q)}(0+)\in(0,q^{-1})$ (X has bounded variation), then the It\^
o--Meyer formula is nothing more than an appropriate version of the
change of variable formula for Stieltjes integrals and the rest of the
proof follows the same line of reasoning as above. The only change
worth mentioning is that the generator of $X$ takes a different form.
Specifically, one has to work with
\[
\Gamma Z^{(q)}(x)=\mathtt{d}Z^{(q)\prime}(x)+\int
_{(-\infty,0)} \bigl(Z^{(q)}(x+y)-Z^{(q)}(x) \bigr)
\Pi(dy),
\]
which satisfies all the required properties by Lemma \ref{generator1}
in the \hyperref[app]{Appendix} and Section 3.2 in \cite{pist}.

This completes the proof of the supermartingale property.\vspace*{9pt}

\textit{Verification of condition} (iii).\quad
The assertion is clear for $(x,s)\in D^*$. Hence, suppose that
$(x,s)\in C^*$. The assertion now follows from the proof of the
supermartingale property (ii). More precisely, replacing $t$
by $t\wedge\tau_{g_\epsilon}$ in (\ref{final}) and recalling that
$(\Gamma-q)Z^{(q)}(y)=0$ for $y>0$ shows that
\[
\E_{x,s} \bigl[e^{-q(t\wedge\tau_{g_\epsilon})}V_\epsilon(X_{t\wedge
\tau_{g_\epsilon}},
\overline X_{t\wedge\tau_{g_\epsilon
}}) \bigr]=V_\epsilon(x,s).
\]
Using that $\tau_{g_\epsilon}<\infty$ a.s. and dominated
convergence, one obtains the desired equality.
\end{pf*}
\begin{pf*}{Proof of Lemma \ref{valuefunction1}}
For $(x,s)\in D^*_I$ we have $T_{\epsilon_1}=0$ so that
\[
\E_{x,s} \bigl[e^{-q(T_{\epsilon_1}\wedge\tau_{\epsilon
_2}^*)+\overline X_{T_{\epsilon_1}\wedge\tau_{\epsilon_2}^*}\wedge
\epsilon_2} \bigr]=e^s=U_{\epsilon_1,\epsilon_2}(x,s).
\]
As for the case $(x,s)\in C^*_I$, write
\begin{eqnarray*}
\E_{x,s} \bigl[e^{-q(T_{\epsilon_1}\wedge\tau_{\epsilon
_2}^*)+\overline X_{T_{\epsilon_1}\wedge\tau_{\epsilon_2}^*}\wedge
\epsilon_2} \bigr]&=&\E_{x,s}
\bigl[e^{-q(T_{\epsilon_1}\wedge\tau
_{\epsilon_2}^*)+\overline X_{T_{\epsilon_1}\wedge\tau_{\epsilon
_2}^*}}1_{\{T_{\epsilon_1}>\tau^+_A\}} \bigr]
\\
&&{}+\E_{x,s} \bigl[e^{-q(T_{\epsilon_1}\wedge\tau_{\epsilon
_2}^*)+\overline X_{T_{\epsilon_1}\wedge\tau_{\epsilon_2}^*}}1_{\{
T_{\epsilon_1}<\tau^+_A\}} \bigr]
\end{eqnarray*}
and denote the first expectation on the right by $I_1$ and the second
expectation by~$I_2$. An application of the strong Markov property at
$\tau^+_A$ and the definition of $V_{\epsilon_2}^*$ (see Theorem \ref
{rwcmainresult}) give
\begin{eqnarray*}
I_1&=&\E_{x,s} \bigl[e^{-q\tau^+_A}1_{\{T_{\epsilon_1}>\tau^+_A\}
}
\bigr]\E_{A,A} \bigl[e^{-q\tau_{\epsilon_2}^*+\overline X_{\tau
_{\epsilon_2}^*}} \bigr]
\\
&=&\frac{W^{(q)}(x-\epsilon_1)}{W^{(q)}(A-\epsilon
_1)}e^AZ^{(q)}\bigl(g_{\epsilon_2}(A)
\bigr).
\end{eqnarray*}
Recalling that $s<g_{\epsilon_2}(A)$ and using the strong Markov
property at $\tau^+_s$ yields
%
\begin{eqnarray}\label{lastexpectation}
I_2&=&e^s\E_{x,s} \bigl[e^{-qT_{\epsilon_1}}1_{\{T_{\epsilon_1}<\tau
^+_s\}}
\bigr]
\nonumber
\\
&&{}+\E_{x,s} \bigl[e^{-q\tau^+_s}1_{\{T_{\epsilon_1}>\tau^+_s\}} \bigr]
\E_{s,s} \bigl[e^{-qT_{\epsilon_1}+\overline X_{T_{\epsilon_1}}}1_{\{
T_{\epsilon_1}<\tau^+_A\}} \bigr]
\nonumber
\\
&=&e^s \biggl(Z^{(q)}(x-\epsilon_1)-W^{(q)}(x-
\epsilon_1)\frac
{Z^{(q)}(s-\epsilon_1)}{W^{(q)}(s-\epsilon_1)} \biggr)
\nonumber\\[-8pt]\\[-8pt]
&&{}+\frac{W^{(q)}(x-\epsilon_1)}{W^{(q)}(s-\epsilon_1)}\E_{s,s} \bigl
[e^{-qT_{\epsilon_1}+\overline X_{T_{\epsilon_1}}}1_{\{T_{\epsilon
_1}<\tau^+_A\}}
\bigr]
\nonumber
\\
&=&e^s \biggl(Z^{(q)}(x-\epsilon_1)-W^{(q)}(x-
\epsilon_1)\frac
{Z^{(q)}(s-\epsilon_1)}{W^{(q)}(s-\epsilon_1)} \biggr)
\nonumber
\\
&&{}+\frac{W^{(q)}(x-\epsilon_1)}{W^{(q)}(s-\epsilon_1)}e^s\E_{0,0} \bigl
[e^{-q\tau^-_{\epsilon_1-s}+\overline X_{\tau^-_{\epsilon
_1-s}}}1_{\{\tau_{\epsilon_1-s}^-<\tau^+_{A-s}\}}
\bigr].\nonumber
\end{eqnarray}

Next, we compute the expectation on the right-hand side of (\ref
{lastexpectation}) by excursion theory. To be more precise, we are
going to make use of the compensation formula of excursion theory, and
hence we shall spend a moment setting up some necessary notation. In
doing so, we closely follow pages 221--223 in \cite{exitproblems} and
refer the reader to Chapters 6 and 7 in \cite{bertoinbook} for
background reading.
The process $L_t:=\overline X_t$ serves as local time at $0$ for the
Markov process $\overline X-X$ under $\P_{0,0}$. Write \mbox{$L^{-1}:=\{
L^{-1}_t\dvtx t\geq0\}$} for the right-continuous inverse of $L$.
The Poisson point process of excursions indexed by local time shall be
denoted by $\{(t,\varepsilon_t)\dvtx t\geq0\}$, where
\[
\varepsilon_t=\bigl\{\varepsilon
_t(s):=X_{L^{-1}_t}-X_{L^{-1}_{t-}+s}\dvtx 0<s<L^{-1}_t-L^{-1}_{t-}
\bigr\},
\]
whenever $L^{-1}_t-L^{-1}_{t-}>0$. Accordingly, we refer to a generic
excursion as $\varepsilon(\cdot)$ (or just $\varepsilon$ for short
as appropriate) belonging to the space $\mathcal{E}$ of canonical
excursions. The intensity measure of the process $\{(t,\varepsilon
_t)\dvtx t\geq0\}$ is given by $dt\times dn$, where $n$ is a measure on the
space of excursions (the excursion measure). A functional of the
canonical excursion that will be of interest is $\overline\varepsilon
=\sup_{s<\zeta}\varepsilon(s)$, where $\zeta(\varepsilon)=\zeta$
is the length of an excursion. A useful formula for this functional
that we shall make use of is the following (cf.~\cite{kyprianou},
equation (8.18)):
%
\begin{equation}\label{propertyppptail}
n(\overline\varepsilon>x)=\frac{W'(x)}{W(x)}
\end{equation}
provided that $x$ is not a discontinuity point in the derivative of $W$
[which is only a concern when $X$ is of bounded variation, but we have
assumed that in this case $\Pi$ is atomless and hence $W$ is
continuously differentiable on $(0,\infty)]$. Another functional that
we will also use is $\rho_a:=\inf\{s>0\dvtx \varepsilon(s)>a\}$, the
first passage time above $a$ of the canonical excursion $\varepsilon$.
We now proceed with the promised calculation involving excursion
theory. Specifically, an application of the compensation formula in the
second equality and using Fubini's theorem in the third equality gives
\begin{eqnarray*}
&&\E\bigl[e^{-q\tau^-_{\epsilon_1-s}+L_{\tau^-_{\epsilon
_1-s}}}1_{\{\tau_{\epsilon_1-s}^-<\tau^+_{A-s}\}} \bigr]
\\
&&\quad=\E\biggl[\sum_{0<t<\infty}e^{-qL^{-1}_{t-}+t}
1_{ \{\overline\varepsilon_{u}\leq u-\epsilon_1+s\ \forall
u<t,t<A-s \}} 1_{\{\overline\varepsilon_t>t-\epsilon_1+s\}}e^{-q\rho
_{t-\epsilon
_1+s}(\varepsilon_t)} \biggr]
\\
&&\quad=\E\biggl[\int_0^{A-s}dt
\,e^{-qL^{-1}_{t}+t}1_{\{\overline
\varepsilon_u\leq u-\epsilon_1+s\ \forall u<t\}}\int_{\mathcal
{E}}1_{\{\overline\varepsilon>t-\epsilon_1+s\}}e^{-q\rho
_{t-\epsilon_1+s}(\varepsilon)}n(d
\varepsilon) \biggr]
\\
&&\quad=\int_0^{A-s} e^{t-\Phi(q)t}\E
\bigl[e^{-qL^{-1}_t+\Phi(q)t}1_{\{
\overline\varepsilon_u\leq u-\epsilon_1+s\ \forall u<t\}} \bigr]\hat
f(t-\epsilon_1+s)
\,dt,
\end{eqnarray*}
where in the first equality the time index runs over local times and
the sum is the usual shorthand for integration with respect to the
Poisson counting measure of excursions, and $\hat f(u)=\frac
{Z^{(q)}(u)W^{(q)\prime}(u)}{W^{(q)}(u)}-qW^{(q)}(u)$ is an expression
taken from Theorem 1 in \cite
{exitproblems}. Next, note that $L^{-1}_t$ is a stopping time and hence
a change of measure according to (\ref{changeofmeasure}) shows that
the expectation inside the integral can be written as
\[
\P^{\Phi(q)} [\overline\varepsilon_u\leq u-
\epsilon_1+s\mbox{ for all }u<t ].
\]
Using the properties of the Poisson point process of excursions
(indexed by local time) and with the help of (\ref{propertyppptail})
and (\ref{scale5}) we may deduce
\begin{eqnarray*}
\P^{\Phi(q)} [\overline\varepsilon_u\leq u-
\epsilon_1+s\mbox{ for all }u<t ]&=&\exp\biggl(-\int
_0^tn_{\Phi(q)}(\overline\varepsilon>u-
\epsilon_1+s)\,du \biggr)
\\
&=&e^{\Phi(q)t}\frac{W^{(q)}(s-\epsilon_1)}{W^{(q)}(t-\epsilon_1+s)},
\end{eqnarray*}
where $n_{\Phi(q)}$ denotes the excursion measure associated with $X$
under $\P^{\Phi(q)}$. By a change of variables and the fact that
$A-\epsilon_1=g_{\epsilon_2}(A)$ we further obtain
\begin{eqnarray*}
&&\E_{0,0} \bigl[e^{-q\tau^-_{\epsilon_1-s}+L_{\tau^-_{\epsilon
_1-s}}}1_{\{\tau_{\epsilon_1-s}^-<\tau^+_{A-s}\}} \bigr]
\\
&&\qquad=W^{(q)}(s-\epsilon_1)e^{\epsilon_1-s}\int
_{s-\epsilon
_1}^{g_{\epsilon_2}(A)} e^{t}\frac{\hat f(t)}{W^{(q)}(t)}\,dt
\\
&&\qquad=-W^{(q)}(s-\epsilon_1)e^{\epsilon_1-s}\int
_{s-\epsilon
_1}^{g_{\epsilon_2}(A)}e^t \biggl(\frac{Z^{(q)}}{W^{(q)}}
\biggr)^\prime(t)\,dt.
\end{eqnarray*}
Integrating by parts on the right-hand side, plugging the resulting
expression into (\ref{lastexpectation}) and finally adding $I_1$ and
$I_2$ gives the result.
\end{pf*}
\begin{pf*}{Proof of Theorem \ref{combinationmainresult}}
Recall that $T_{\epsilon_1}:=\inf\{t\geq0\dvtx X_t\leq\epsilon
_1\}$ and from Lemma \ref{valuefunction1} that, for $(x,s)\in E$,
%
\begin{equation}\label{cond6}
V_{\epsilon_1,\epsilon_2}(x,s)=\E_{x,s} \bigl[e^{-q(T_{\epsilon
_1}\wedge\tau_{\epsilon_2}^*)+\overline X_{T_{\epsilon_1}\wedge
\tau_{\epsilon_2}^*}\wedge\epsilon_2} \bigr].
\end{equation}
Similarly to the proof of Theorem \ref{rwcmainresult}, it is now
enough to prove that:
\begin{longlist}[(ii)]
\item[(i)] $V_{\epsilon_1,\epsilon_2}(x,s)\geq e^{s\wedge
\epsilon_2}$ for all $(x,s)\in E_{\epsilon_1}$;
\item[(ii)] $\{e^{-q(t\wedge T_{\epsilon_1})}V_{\epsilon
_1,\epsilon_2}(X_{t\wedge T_{\epsilon_1}},\overline X_{t\wedge
T_{\epsilon_1}})\dvtx t\geq0\}$ is a right-continuous
$\P_{x,s}$-su\-per\-mar\-tin\-gale for all $(x,s)\in E_{\epsilon_1}$.
\end{longlist}
Condition (i) is clearly satisfied, so we devote
the remainder of this proof to checking condition (ii).\vspace*{9pt}

\textit{Supermartingale property} (ii).\quad
Let $Y_t:=e^{-qt}V_{\epsilon_1,\epsilon_2}(X_t,\overline X_t)$ for
$t\geq0$. Analogously to the proof of Theorem \ref{rwcmainresult},
it suffices to show that for $(x,s)\in E_{\epsilon_1}$ we have the inequality
%
\begin{equation}\label{inequality2}
\E_{x,s} [Y_{t\wedge T_{\epsilon_1}} ]\leq V_{\epsilon
_1,\epsilon_2}(x,s).
\end{equation}
The latter is clear for $(x,s)\in D^*_I$. If $(x,s)\in
C_{\mathit{II}}^*\cup D^*_{\mathit{II}}$, inequality (\ref{inequality2}) can be
extracted from the proof of Theorem \ref{rwcmainresult} where it is
shown that the process $(e^{-qt}V_{\epsilon_2}^*(X_t,\overline
X_t))_{t\geq0}$ is a $\P_{x,s}$-supermartinagle for all
$(x,s)\in E$. In particular, the process $(Y_t)_{t\geq0}$ is a $\P
_{x,s}$-supermartingale for $(x,s)\in C_{\mathit{II}}^*\cup D^*_{\mathit{II}}$.
The supermartingale property is preserved when stopping at $T_{\epsilon
_1}$ and therefore we obtain, for $(x,s)\in C_{\mathit{II}}^*\cup D^*_{\mathit{II}}$,
%
\begin{equation}\label{partofinequality}
\E_{x,s} [Y_{t\wedge T_{\epsilon_1}} ]\leq V_{\epsilon
_1,\epsilon_2}(x,s).
\end{equation}
Thus, it remains to establish (\ref{inequality2}) for $(x,s)\in
C_{I}^*$. To this end, we first prove that the process $
(Y_{t\wedge T_{\epsilon_1}\wedge\tau_{\epsilon_2}^*} )_{t\geq
0}$ is a $\P_{x,s}$-martingale. The strong Markov property gives
%
\begin{eqnarray}\label{split}
\E_{x,s} [Y_{T_{\epsilon_1}\wedge\tau_{\epsilon_2}^*} \vert\mathcal{F}_t
]&=&Y_{T_{\epsilon_1}\wedge\tau_{\epsilon
_2}^*}1_{\{T_{\epsilon_1}\wedge\tau_{\epsilon_2}^*\leq t\}}
\nonumber\\[-8pt]\\[-8pt]
&&{}+e^{-qt}\E_{X_t,\overline X_t} [Y_{T_{\epsilon_1}\wedge\tau
_{\epsilon_2}^*} ]1_{\{T_{\epsilon_1}\wedge\tau_{\epsilon
_2}^*>t\}}.\nonumber
\end{eqnarray}
By definition of $V_{\epsilon_1,\epsilon_2}$ we see that
\[
Y_{T_{\epsilon_1}\wedge\tau_{\epsilon_2}^*}=\cases{\exp(-qT_{\epsilon
_1}+\overline X_{T_{\epsilon_1}}
), &\quad on $\bigl\{T_{\epsilon_1}\leq\tau_{\epsilon_2}^*\bigr\}$,
\cr
\exp
\bigl(-q\tau_{\epsilon_2}^*+\overline X_{\tau_{\epsilon_2}^*}\bigr),
&\quad
on $\bigl
\{T_{\epsilon_1}>\tau_{\epsilon_2}^*\bigr\}$,}
\]
which shows that the second term on the right-hand side of (\ref
{split}) equals
\begin{eqnarray*}
&&e^{-qt}\E_{X_t,\overline X_t} \bigl[e^{-q(T_{\epsilon_1}\wedge\tau
_{\epsilon_2}^*)+\overline X_{T_{\epsilon_1}\wedge\tau_{\epsilon
_2}^*}}
\bigr](1_{\{t\leq\tau^+_A\}}+1_{\{t>\tau^+_A\}})1_{\{
T_{\epsilon_1}\wedge\tau_{\epsilon_2}^*>t\}}
\\
&&\qquad= \bigl(e^{-qt}U_{\epsilon_1,\epsilon_2}(X_t,\overline
X_t)1_{\{
t\leq\tau^+_A\}}+e^{-qt}V_{\epsilon_2}^*(X_t,
\overline X_t)1_{\{
t>\tau^+_A\}} \bigr)1_{\{T_{\epsilon_1}\wedge\tau_{\epsilon_2}^*>t\}
}
\\
&&\qquad=e^{-qt}V_{\epsilon_1,\epsilon_2}(X_t,\overline
X_t)1_{\{
T_{\epsilon_1}\wedge\tau_{\epsilon_2}^*>t\}}
\\
&&\qquad=Y_t1_{\{T_{\epsilon_1}\wedge\tau_{\epsilon_2}^*>t\}}.
\end{eqnarray*}
Thus, $\E_{x,s} [Y_{T_{\epsilon_1}\wedge\tau_{\epsilon
_2}^*} \vert\mathcal{F}_t ]=Y_{t\wedge T_{\epsilon_1}\wedge
\tau_{\epsilon_2}^*}$ which implies the martingale property of $
(Y_{t\wedge T_{\epsilon_1}\wedge\tau_{\epsilon_2}^*} )_{t\geq
0}$. Again using the strong Markov property we further obtain for
$(x,s)\in C^*_I$,
\begin{eqnarray*}
\E_{x,s} [Y_{t\wedge T_{\epsilon_1}} \vert\mathcal{F}_{\tau
_{\epsilon_2}^*}
]&=&Y_{t\wedge T_{\epsilon_1}}1_{\{t\wedge
T_{\epsilon_1}\leq\tau_{\epsilon_2}^*\}}
\\
&&{}+e^{-q\tau_{\epsilon_2}^*}\E_{X_{\tau_{\epsilon_2}^*},\overline
X_{\tau_{\epsilon_2}^*}} [Y_{(t-u)\wedge T_{\epsilon_1}} ] \vert_{u=\tau
_{\epsilon_2}^*}1_{\{t\wedge T_{\epsilon_1}>\tau
_{\epsilon_2}^*\}}
\\
&\leq&Y_{t\wedge T_{\epsilon_1}}1_{\{t\wedge T_{\epsilon_1}\leq\tau
_{\epsilon_2}^*\}}+e^{-q\tau_{\epsilon_2}^*}V_{\epsilon_1,\epsilon
_2}(X_{\tau_{\epsilon_2}^*},
\overline X_{\tau_{\epsilon_2}^*})1_{\{
t\wedge T_{\epsilon_1}>\tau_{\epsilon_2}^*\}}
\\
&=&Y_{t\wedge T_{\epsilon_1}\wedge\tau_{\epsilon_2}^*},
\end{eqnarray*}
where\vspace*{1pt} the inequality follows from (\ref{partofinequality}) and the
fact that $(X_{\tau_{\epsilon_2}^*}\overline X_{\tau_{\epsilon
_2}^*})\in D^*_{\mathit{II}}$ on \mbox{$\{t\wedge T_{\epsilon_1}>\tau_{\epsilon
_2}^*\}$}. Thus, $\E_{x,s} [Y_{t\wedge T_{\epsilon_1}} ]\leq
U_{\epsilon_1,\epsilon_2}(x,s)=V_{\epsilon_1,\epsilon_2}(x,s)$ for
$(x,s)\in C^*_I$. This completes the proof.
\end{pf*}
\begin{pf*}{Proof of Corollary \ref{consistencybarrier}} Part (a)
follows from the proof of Theorem \ref
{combinationmainresult} by replacing $g_\epsilon$ with $g_\infty
(s)=k^*$ and $A$ by $\epsilon_1+k^*$. For part (b), let
$\epsilon_1\in\R$ be given and recall that due to the assumption
$q\leq\psi(1)$ we have $\lim_{s\downarrow-\infty}g_{\epsilon
_1}(s)=\infty$. For an arbitrary $\delta>\epsilon_1$, the uniqueness
in Lemma \ref{g} implies that
\[
g_\delta(s)=g_{\epsilon_1}(s-\delta+\epsilon_1),\qquad s\in(-
\infty,\delta).
\]
It follows that $\lim_{\delta\uparrow\infty}g_{\delta}(s)=\infty$
for $s\in\R$ and that $\lim_{\delta\uparrow\infty}g_\delta
(A_\delta)=\infty$. Hence, for $(x,s)\in E_{\epsilon_1}$, we have
\[
V^*_{\epsilon_1,\infty}(x,s):=\sup_{\tau\in\mathcal{M}_{\epsilon
_1}}\E_{x,s}
\bigl[e^{-q(T_{\epsilon_1}\wedge\tau)+\overline
X_{T_{\epsilon_1}\wedge\tau}} \bigr]\geq\lim_{\delta\uparrow
\infty}V^*_{\epsilon_1,\delta}(x,s)=
\infty.
\]
On the other hand, if $(x,s)\in E\setminus E_{\epsilon_1}$, then
clearly $V_{\epsilon_1,\infty}^*(x,s)=e^s$. This completes the proof.
\end{pf*}
%
\section{Examples}\label{applications}
The solutions of (\ref{problem1}) and (\ref{problem3}) are given
semi-explicitly in terms of scale functions and a specific solution
$g_\epsilon$ and $g_{\epsilon_2}$, respectively, of the ordinary
differential equation (\ref{diffequ}). The aim of this section is to
look at some examples where the solutions of (\ref{problem1})
and (\ref{problem3}) can be computed more explicitly. For simplicity,
we will assume from now on that every spectrally negative L\'evy
process $X$ considered below is such that $q>\psi(1)$ and
$W^{(q)}(0+)<q^{-1}$. Also assume to begin with that there is an upper cap
$\epsilon$ only.

%
\begin{figure}

\includegraphics{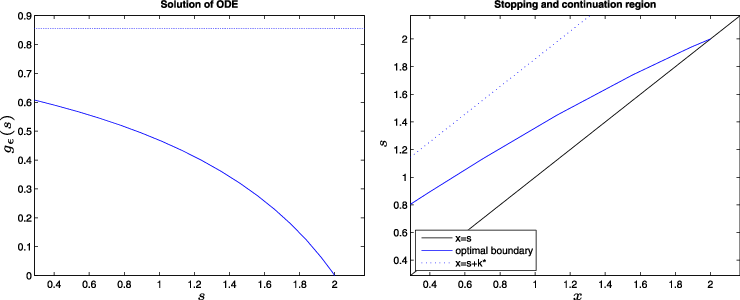}

\caption{An illustration of $s\mapsto g_\epsilon(s)$ and the
corresponding optimal boundary for $q=1.6$, $\epsilon=2$, $\sigma=0$,
$\mu=3$, $a=3$ and $\rho=0.1$.}\label{num1}
\end{figure}

A first step towards more explicit solutions of (\ref {problem1}) is
looking at processes $X$ where explicit expressions for $W^{(q)}$ and
$Z^{(q)}$ are available. In recent years various authors have found
several processes whose scale functions are explicitly known (Example
1.3, Chapter 4 and Section 5.5 in \cite{KuzKypRiv}, e.g.). Here,
however, we would additionally like to find $g_\epsilon$ explicitly. To
the best of our knowledge, we do not know of any examples where this is
possible. One might instead try to solve (\ref{diffequ}) numerically,
but this is not straightforward as there is no initial point to start a
numerical scheme from and, moreover, the possibility of $g_\epsilon$
having infinite gradient at $\epsilon$ might lead to inaccuracies in
the numerical scheme. Therefore, we follow a different route which
avoids these difficulties. Instead of looking at $g_\epsilon$, we
rather focus on its inverse
%
\begin{equation}\label{num}
H(s)=\epsilon+\int_0^s \biggl(1-
\frac{Z^{(q)}(\eta)}{qW^{(q)}(\eta
)} \biggr)^{-1}\,d\eta,\qquad s\in\bigl(0,k^*\bigr),
\end{equation}
where $k^*\in(0,\infty)$ is the unique root of
$Z^{(q)}(z)-qW^{(q)}(z)=0$. It turns out that in some cases (including
the Black--Scholes model) $H$ can be computed explicitly. Since $H$ is
the inverse of $g_\epsilon$, plotting $(H(y),y),y\in(0,k^*)$, yields
visualisations of $s\mapsto g_\epsilon(s)$ for
$s\in(-\infty,\epsilon)$; see Figures \ref{num1}--\ref{num3}.
%
\begin{figure}

\includegraphics{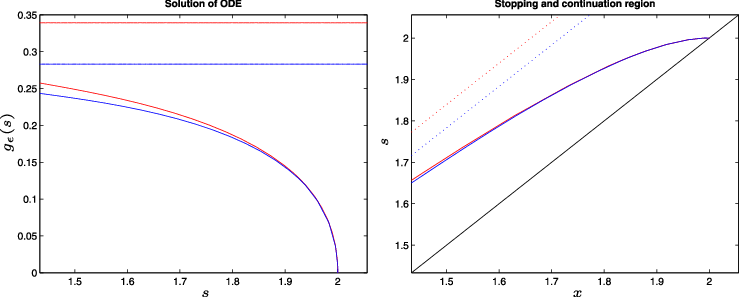}

\caption{Left: a visualization of $s\mapsto g_\epsilon(s)$ for when
$q=4$, $\epsilon=2$, $\sigma=1$ and $\mu=2$ (red) and $q=4$,
$\epsilon=2$, $\sigma=1$, $\mu=2$, $a=3$ and $\rho=0.1$ (blue).
Right: an illustration of the corresponding optimal
boundaries.}\label{num2}\vspace*{-3pt}
\end{figure}
Similarly, plotting $(H(y)-y,H(y)),y\in(0,k^*)$, produces
visualisations of the optimal stopping boundary in the $(x,s)$-plane;
see Figures \ref{num1}--\ref{num3}. Unfortunately, it is often the
%
\begin{figure}[b]

\includegraphics{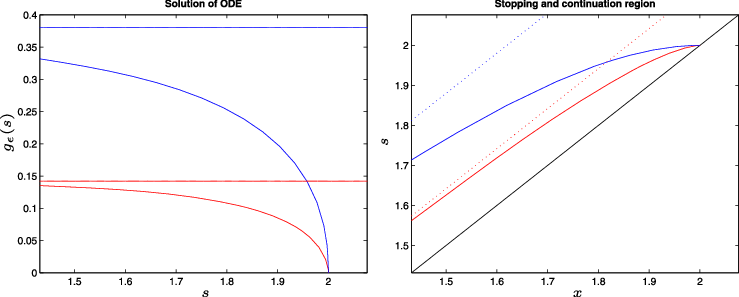}

\caption{Left: a visualisation of $s\mapsto g_\epsilon(s)$ when $q=2$
and $\epsilon=2$, and $X$ is either a linear Brownian motion (blue
curve, $\sigma=\sqrt{2}$, $\mu=0$) or an $\alpha$-stable process
(red curve, $\alpha=1.6$).}\label{num3}
\end{figure}
case that we cannot compute the integral in (\ref{num}) explicitly in
which case one might use numerical integration in Matlab to obtain an
approximation of the integral. The procedure just described is carried
out below for different examples of $X$.

\subsection{Brownian motion with drift and compound Poisson jumps}
Consider the process
\[
X_t=\sigma W_t+\mu t-\sum_{i=1}^{N_t}
\xi_i,\qquad t\geq0,
\]
where $\sigma>0$, $\mu\in\R$, $(W_t)_{t\geq0}$ is a standard
Brownian motion, $(N_t)_{t\geq0}$ is a Poisson process with intensity\vadjust{\goodbreak}
$a>0$ and $\xi_i$ are i.i.d. random variables which are exponentially
distributed with parameter $\rho>0$. The processes $(W_t)_{t\geq0}$
and $ (N_t)_{t\geq0}$ as well as the sequence $(\xi_i)_{i\in\N}$
are assumed to be mutually independent. The Laplace exponent of $X$ is
given by
\[
\psi(\theta)=\frac{\sigma^2}{2}\theta^2+\mu\theta-
\frac{a\theta
}{\rho+\theta},\qquad \theta\geq0.
\]
It is known (cf. Example 1.3 in \cite{KuzKypRiv} and Section 8.2
of \cite{exitproblems}) that
%
\begin{equation}\label{Wscale}
W^{(q)}(x)=\frac{e^{\Phi(q)x}}{\psi^\prime(\Phi(q))}+\frac
{e^{-\zeta_1x}}{\psi^\prime(-\zeta_1)}+\frac{e^{-\zeta_2x}}{\psi
^\prime(-\zeta_2)},\qquad x
\geq0,
\end{equation}
where $-\zeta_2<-\rho<-\zeta_1<0<\Phi(q)$ are the three real
solutions of the equation $\psi(\theta)=q$, and that, for $x\geq0$,
%
\begin{equation}\label{Zscale}
Z^{(q)}(x)=D_1e^{\Phi(q)x}+D_2e^{-\zeta_1x}+D_3e^{-\zeta_2x},
\end{equation}
where $D_1=\frac{q}{\Phi(q)\psi^\prime(\Phi(q))}$, $D_2=\frac
{q}{-\zeta_1\psi^\prime(-\zeta_1)}$ and $D_3=\frac{q}{-\zeta
_2\psi^\prime(-\zeta_2)}$.\vspace*{1pt}

As a first example consider $\sigma=0$. In this case $\psi
(\theta)=q$ reduces to a quadratic equation, and one can calculate explicitly
\begin{eqnarray*}
\zeta_1&=&\frac{1}{2\mu} \bigl(\sqrt{(a+q-\mu
\rho)^2+4\mu q\rho}-(a+q-\mu\rho) \bigr),
\\
\Phi(q)&=&\frac{1}{2\mu} \bigl(\sqrt{(a+q-\mu\rho)^2+4\mu q\rho
}+(a+q-\mu\rho) \bigr).
\end{eqnarray*}
Moreover, it follows that
\[
k^*=\frac{1}{\zeta_1+\phi(q)}\log\biggl(\frac{\Phi(q)\psi^\prime
(\Phi(q))(\zeta_1+1)}{\zeta_1\psi^\prime(-\zeta_1)(1-\Phi
(q))} \biggr).
\]
Using elementary algebra and integration one finds, for $s\in(0,k^*)$,
\begin{eqnarray*}
H(s)&=&\epsilon+\int_0^s \biggl(
\frac{D_1\Phi(q)e^{(\Phi(q)+\zeta
_1)x}}{D_1(\Phi(q)-1)e^{(\Phi(q)+\zeta_1)x}-D_2(\zeta_1+1)} \biggr)\, dx
\\
&&{}-\int_0^s\frac{D_2\zeta_1e^{-(\zeta_1+\Phi(q))x}}{D_1(\Phi
(q)-1)-D_2(\zeta_1+1)e^{-(\zeta_1+\Phi(q))x}}\,dx
\\
&=&\epsilon+\int_0^s \biggl(
\frac{\Phi(q)e^{Ax}}{Be^{Ax}-CD}-\frac
{\zeta_1e^{-Ax}}{C^{-1}B-De^{-Ax}} \biggr)\,dx
\\
&=&\epsilon+\frac{\Phi(q)}{AB}\log\biggl\vert\frac
{Be^{As}-CD}{B-CD} \biggr\vert-
\frac{\zeta_1}{AD}\log\biggl\vert\frac{B-CDe^{-As}}{B-CD} \biggr\vert,
\end{eqnarray*}
where\vspace*{1pt} $A:=\zeta_1+\Phi(q),B:=\Phi(q)-1,C:=\frac{\Phi(q)\psi
^\prime(\Phi(q))}{-\zeta_1\psi^\prime(-\zeta_1)}$ and $D:=\zeta
_1+1$. An example for a certain choice of parameters is given in
Figure \ref{num1}.

Next, assume $\sigma>0$ and $\rho=\infty$; that is, $X$ is a linear
Brownian motion. In particular, this includes the Black--Scholes model.
Again, as explained in Example~1.3 of \cite{KuzKypRiv}, the equation
$\psi(\theta)=q$ reduces to a quadratic equation and $\zeta_1=\delta
-\gamma$ and $\Phi(q)=\delta+\gamma$, where
\[
\gamma:=-\frac{\mu}{\sigma^2} \quad\mbox{and}\quad \delta:=\frac
{1}{\sigma^2}\sqrt{
\mu^2+2q\sigma^2}.
\]
Furthermore, (\ref{Wscale}) and (\ref{Zscale}) may be rewritten on
$x\geq0$ as
%
\begin{eqnarray}\label{below1}
W^{(q)}(x)&=&\frac{2}{\sigma^2\delta}e^{\gamma x}\sinh(\delta x)
\quad\mbox{and}\nonumber\\[-8pt]\\[-8pt]
Z^{(q)}(x)&=&e^{\gamma x}\cosh(\delta x)-\frac{\gamma
}{\delta}e^{\gamma x}
\sinh(\delta x)\nonumber
\end{eqnarray}
and one can compute
%
\begin{equation}\label{below2}
k^*=\frac{1}{\Phi(q)+\zeta_1}\log\biggl(\frac{1+\zeta
_1^{-1}}{1-\Phi(q)^{-1}} \biggr).
\end{equation}
Using elementary algebra in the first and formula 2.447.1 of \cite
{gradshteyn} in the second equality one obtains, for $s\in(0,k^*)$,
\begin{eqnarray*}
H(s)&=&\epsilon+\frac{2q}{\sigma^2\delta}\int_0^{s\delta}
\frac
{\sinh(x)}{(2q/\sigma^2+\gamma)\cosh(x)-\delta\sinh(x)}\,dx
\\
&=&\epsilon+\frac{2q}{\sigma^2\delta(F^2-\delta^2 )} \biggl(F\delta
s-\delta\log\biggl\vert
\frac{\sinh(\tanh
^{-1}(-\delta F^{-1}) )}{\sinh(\delta s+\tanh^{-1}(-\delta
F^{-1}) )} \biggr\vert\biggr),
\end{eqnarray*}
where $F:=2q/\sigma^2+\gamma$. An example for a certain parameter
choice is provided in Figure \ref{num2}.

In the next example we combine the first example with the
second one. More precisely, suppose that $\sigma>$ and $\rho\in
(0,\infty)$, that is, a linear Brownian motion with exponential jumps.
In this case we are unable to compute $k^*$ and $H$ explicitly. We
therefore find $k^*$ numerically and use numerical integration to
obtain an approximation of $k^*$ and $H$, respectively; see Figure \ref{num2}.

\subsection{Stable jumps}
Suppose that $X$ is an $\alpha$-stable process, where $\alpha
\in(1,2]$ with Laplace exponent
$\psi(\theta)=\theta^\alpha,\theta\geq0$. It is known (cf. Example 4.17
of \cite{KuzKypRiv} and Section 8.3 of \cite{exitproblems}) that,
for $x\geq0$,
\[
W^{(q)}(x)=x^{\alpha-1}E_{\alpha,\alpha}\bigl(qx^\alpha
\bigr) \quad\mbox{and}\quad Z^{(q)}(x)=E_{\alpha,1}\bigl(qx^{\alpha}
\bigr),
\]
where $E_{\alpha,\beta}$ is the two-parameter Mittag--Leffler function
which is defined for $\alpha>0,\beta>0$ as
\[
E_{\alpha,\beta}(x)=\sum_{n=0}^\infty
\frac{x^n}{\Gamma(\alpha
n+\beta)}.
\]
Again, using numerical integration and a Matlab function that computes
the Mittag--Leffler function (cf. \cite{mittaglefflerfunction}) one
may approximate $k^*$ and $H$, respectively; see Figure \ref{num3}.
Additionally, we have computed the value function for a choice of
parameters (Figure \ref{num4}).

%
\begin{figure}

\includegraphics{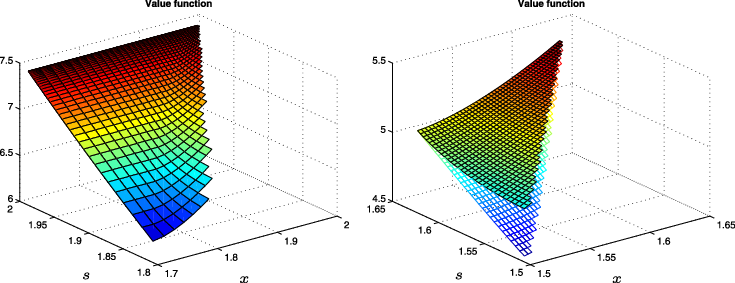}

\caption{Left: a visualisation of $V_\epsilon^*(x,s)$ when $X$ is
$\alpha$-stable with parameter choice $q=3$, $\epsilon=2$ and
$\alpha=1.6$. Right: an illustration of the difference between
$V_{\epsilon_2}^*(x,s)$ (darker surface) and
$V^*_{\epsilon_1,\epsilon_2}(x,s)$ (lighter surface) on
$C^*_{I,\epsilon_1,\epsilon_2}$ for the\vspace*{1pt} same $X$ and same parameters as
on the left. In this case $A\approx1.63$, where $A$ is formally
defined in Section \protect\ref{Mpwualc}.}\label{num4}
\end{figure}

If one considers a lower cap $\epsilon_1$ and an upper cap $\epsilon
_2$, then the only thing that changes for the optimal boundary is that
one has to include an additional vertical line at the value of the
lower cap $\epsilon_1$. However, introducing a lower cap will make a
difference, that is, the value\vspace*{1pt} functions $V_{\epsilon_2}^*(x,s)$ and
$V^*_{\epsilon_1,\epsilon_2}(x,s)$ will be different for $(x,s)\in
C^*_{I,\epsilon_1,\epsilon_2}$;\vspace*{1pt} see Theorems \ref{rwcmainresult}
and \ref{combinationmainresult}. Exploiting the fact that $H$ is the
inverse of $g_{\epsilon_2}$ in a similar way as above, one may also
obtain numerical approximations of the value functions $V^*_{\epsilon
_2}(x,s)$ and $V^*_{\epsilon_1,\epsilon_2}(x,s)$; see Figure \ref{num4}.

\subsection{Maximum process with lower cap only}
Assume the same setting as in the second example above, that is,
$X_t=\sigma W_t+\mu t$. The scale functions and $k^*$ are given
by (\ref{below1}) and (\ref{below2}), respectively. If we suppose that
there is a lower cap $\epsilon_1\in\R$ and no upper cap ($\epsilon
_2=\infty$), then Corollary \ref{consistencybarrier} can be
rewritten more explicitly as follows.
%
\begin{lem}\label{explicitbarrier}
The $V^*$ and $U_{\epsilon_1,\infty}$ part of the optimal value
function $V_{\epsilon_1,\infty}^*$ are given by
\[
V^*(x,s)=\frac{1}{\Phi(q)+\zeta_1} \biggl(\Phi(q) \biggl(\frac
{e^x}{e^{s-k^*}}
\biggr)^{-\zeta_1}+\zeta_1 \biggl(\frac
{e^x}{e^{s-k^*}}
\biggr)^{\Phi(q)} \biggr)
\]
and
\begin{eqnarray*}
U_{\epsilon_1,\infty}(x,s)&=& \biggl(\frac{e^x}{e^{\epsilon_1}} \biggr
)^{-\zeta_1}
\biggl[-\frac{e^{\epsilon_1}}{\beta} \biggl(\int_{\beta
(s-\epsilon_1)}^{\beta k^*}
\frac{e^{u(1+y)}}{e^{u}-1}\,du-e^{k^*\Phi
(q)} \biggr) \biggr]
\\
&&{}+ \biggl(\frac{e^x}{e^{\epsilon_1}} \biggr)^{\Phi(q)} \biggl[\frac
{e^{\epsilon_1}}{\beta}
\biggl(\int_{\beta(s-\epsilon_1)}^{\beta
k^*}\frac{e^{uy}}{e^{u}-1}
\,du-e^{-k^*\zeta_1} \biggr) \biggr],
\end{eqnarray*}
where $\beta=\Phi(q)+\zeta_1=2\delta$ and $y=\beta^{-1}$.
\end{lem}

The proof of this result is a lengthy computation provided in
Appendix~\ref{alengthy}. Finally, if we set $\epsilon_1=\epsilon$,
$\mu=r-\sigma^2/2$ for some $r\geq0$ and $q=\lambda+r$ for some
$\lambda>0$ we recover Theorem 3.1 of \cite{abarrierversion}.

\begin{appendix}\label{app}
\section{Complementary results on the infinitesimal generator of $X$}
In this section we provide some results concerning the infinitesimal
generator of $X$ when applied to the scale function $Z^{(q)}$.

First assume that $X$ is of unbounded variation, and define an operator
$(\Gamma,\mathcal{D} (\Gamma))$ as follows. $\mathcal{D}(\Gamma)$
stands for the family of functions $f\in C^2(0,\infty)$ such that the integral
\[
\int_{(-\infty,0)} \bigl(f(x+y)-f(x)-yf^\prime(x)1_{\{y\geq-1\}}
\bigr) \Pi(dy)
\]
is absolutely convergent for all $x>0$. For any $f\in\mathcal
{D}(\Gamma)$, we define the function $\Gamma f\dvtx (0,\infty)\rightarrow
\R$ by
\begin{eqnarray*}
\Gamma f(x)&=&-\gamma f'(x)+\frac{\sigma^2}{2}f''(x)\\
&&{}+
\int_{(-\infty,0)} \bigl(f(x+y)-f(x)-yf'(x)1_{\{y\geq-1\}}
\bigr) \Pi(dy).
\end{eqnarray*}
Similarly, if $X$ is of bounded variation, then $\mathcal{D}(\Gamma)$
stands for the family of \mbox{$f\in C^1(0,\infty)$} such that the integral
\[
\int_{(-\infty,0)} \bigl(f(x+y)-f(x) \bigr) \Pi(dy)
\]
is absolutely convergent for all $x>0$, and for $f\in\mathcal
{D}(\Gamma)$, we define the function $\Gamma f\dvtx (0,\infty)\rightarrow
\R$ by
\[
\Gamma f(x)=\mathtt{d}f'(x)+\int_{(-\infty,0)}
\bigl(f(x+y)-f(x) \bigr)\Pi(dy).
\]
In the sequel it should always be clear from the context in which of
the two cases we are and therefore there should be no ambiguity when
writing $\mathcal{D}(\Gamma)$ and $\Gamma$.
%
\begin{lem}\label{generator1}
We have that $Z^{(q)}\in\mathcal{D}(\Gamma)$ and the function
$x\mapsto\Gamma Z^{(q)}(x)$ is continuous on $(0,\infty)$.
\end{lem}
\begin{pf}
We prove the unbounded and bounded variation case separately.\eject

\textit{Unbounded variation}:
To show\vspace*{1pt} that $Z^{(q)}\in\mathcal{D}(\Gamma)$ it is enough to check
that the integral part of $\Gamma Z^{(q)}$ is absolutely convergent
since $Z^{(q)}\in C^2(0,\infty)$. Fix\vspace*{1pt} $x>0$ and write the integral
part of $\Gamma Z^{(q)}$ as
\begin{eqnarray*}
&&\int_{(-\infty,-\delta)} \bigl\vert Z^{(q)}(x+y)-Z^{(q)}(x)-yZ^{(q)\prime
}(x)1_{\{y\geq-1\}}
\bigr\vert\Pi(dy)
\\
&&\qquad{}+\int_{(-\delta,0)} \bigl\vert Z^{(q)}(x+y)-Z^{(q)}(x)-yZ^{(q)\prime
}(x)1_{\{
y\geq-1\}}
\bigr\vert\Pi(dy),
\end{eqnarray*}
where the value $\delta=\delta(x)\in(0,1)$ is chosen such that
$x-\delta>0$. For $y\in(-\infty,-\delta)$ the monotonicity
of $Z^{(q)}$ implies
%
\begin{equation}\label{estimate1}
\bigl\vert Z^{(q)}(x+y)-Z^{(q)}(x)-yZ^{(q)\prime}(x)1_{\{y\geq-1\}}
\bigr\vert\leq2Z^{(q)}(x)+Z^{(q)\prime}(x)
\end{equation}
and for $y\in(-\delta,0)$, using the mean value theorem, we have
%
\begin{eqnarray}\label{estimate2}
&&\bigl\vert Z^{(q)}(x+y)-Z^{(q)}(x)-yZ^{(q)\prime}(x)\bigr\vert
\nonumber
\\
&&\qquad=q\vert y\vert\bigl\vert W^{(q)}\bigl(\xi(y)\bigr)-W^{(q)}(x)
\bigr\vert\qquad\mbox{where $\xi(y)\in(x+y,x)$}
\nonumber\\[-8pt]\\[-8pt]
&&\qquad=q\vert y\vert\biggl\vert\int_{\xi(y)}^xW^{(q)\prime}(z)
\, dz \biggr\vert
\nonumber
\\
&&\qquad\leq qy^2\sup_{z\in[x-\delta,x]}W^{(q)\prime}(z).\nonumber
\end{eqnarray}
Using these two estimates and defining
$C(\delta)=\int_{(-\delta,0)}y^2\Pi(dy)<\infty$, we see that
\begin{eqnarray*}
&&\int_{(-\infty,0)} \bigl\vert Z^{(q)}(x+y)-Z^{(q)}(x)-yZ^{(q)\prime
}(x)1_{\{
y\geq-1\}}
\bigr\vert\Pi(dy)
\\
&&\qquad\leq\bigl(2Z^{(q)}(x)+Z^{(q)\prime}(x) \bigr)\Pi(-\infty,-\delta
)+qC(\delta)\sup_{z\in[x-\delta,x]}W^{(q)\prime}(z)<\infty.
\end{eqnarray*}

For continuity, let $x>0$ and choose $\delta=\delta(x)\in(0,1)$ such
that $x-2\delta>0$ as well as a sequence $(x_n)_{n\in\N}$ converging
to $x$. Moreover, let $n_0\in\N$ such that for all $n\geq n_0$ we
have $\vert x_n-x\vert<\delta$. In particular, it holds that
$x_n-\delta>0$ for $n\geq n_0$ and hence, using the estimates in (\ref
{estimate1}) and (\ref{estimate2}), we have for all $n\geq n_0$
\begin{eqnarray*}
\hspace*{-3pt}&&\bigl\vert Z^{(q)}(x_n+y)-Z^{(q)}(x_n)-yZ^{(q)\prime}(x_n)1_{\{y\geq-1\}
}
\bigr\vert
\\
\hspace*{-3pt}&&\qquad\leq qy^2\sup_{z\in[x_n-\delta,x_n]}W^{(q)\prime}(z)1_{\{y\geq
-\delta\}
}+
\bigl(2Z^{(q)}(x_n)+Z^{(q)\prime}(x_n)
\bigr)1_{\{y<-\delta\}}
\\
\hspace*{-3pt}&&\qquad\leq qy^2\sup_{z\in[x-2\delta,x+\delta]}W^{(q)\prime}(z)1_{\{
y\geq
-\delta\}}+
\bigl(2Z^{(q)}(x+\delta)+Z^{(q)\prime}(x+\delta)
\bigr)1_{\{
y<-\delta\}}.
\end{eqnarray*}
Since the last term is $\Pi$-integrable, the continuity assertion
follows by dominated convergence and the fact that $Z^{(q)}\in
C^2(0,\infty)$.\eject

\textit{Bounded variation}:
To show that $Z^{(q)}\in\mathcal{D}(\Gamma)$ it is enough to show
that the integral part of $\Gamma Z^{(q)}$ is absolutely convergent
since $Z^{(q)}\in C^1(0,\infty)$. Using the monotonicity and the
definition of $Z^{(q)}$, it is easy to see that for fixed $x>0$,
\begin{eqnarray*}
&&\int_{(-\infty,0)}\bigl\vert Z^{(q)}(x+y)-Z^{(q)}(x)
\bigr\vert\Pi(dy)
\\
&&\qquad\leq2Z^{(q)}(x)\Pi(-\infty,-1)+qW^{(q)}(x)\int
_{(-1,0)}\vert y\vert\Pi(dy)<\infty.
\end{eqnarray*}
The continuity assertion follows in a straightforward manner from
dominated convergence and the fact that $Z^{(q)}\in C^1(0,\infty)$.
\end{pf}
%
%
\section{A lengthy computation}\label{alengthy}
\begin{pf*}{Proof of Lemma \ref{explicitbarrier}}
The first part is a short calculation using the definition of $\gamma
$, $\delta$, $\zeta_1$, $\Phi(q)$ and that $\cosh(z)=\frac
{e^z+e^{-z}}{2}$ and $\sinh(z)=\frac{e^z-e^{-z}}{2}$. As for the
second part, recall that, for $(x,s)\in C^*_I\cup D^*_I$,
\[
U_{\epsilon_1,\infty}(x,s)=e^sZ^{(q)}(x-\epsilon_1)+e^{\epsilon
_1}W^{(q)}(x-
\epsilon_1)\int_{s-\epsilon_1}^{k^*}e^t
\frac
{Z^{(q)}(t)}{W^{(q)}(t)}\,dt.
\]
It is easy to see that
\[
e^t\frac{Z^{(q)}(t)}{W^{(q)}(t)}=e^t\frac{\delta\sigma^2}{2} \biggl(
\frac{1}{1-e^{-2\delta t}}+\frac{1}{e^{2\delta t}-1} \biggr)-e^t\frac
{\gamma\sigma^2}{2},
\]
which, after a change of variables, gives
\begin{eqnarray*}
\int_{s-\epsilon_1}^{k^*}e^t\frac{Z^{(q)}(t)}{W^{(q)}(t)}\,
dt&=&\frac
{\sigma^2}{4} \biggl(\int_{\beta(s-\epsilon_1)}^{\beta k^*}
\frac
{e^{u(1+y)}}{e^u-1}\,du+\int_{\beta(s-\epsilon_1)}^{\beta k^*}
\frac
{e^{uy}}{e^u-1}\,du \biggr)
\\
&&{}+\frac{\gamma\sigma^2}{2}\bigl(e^{s-\epsilon_1}-e^{k^*}\bigr),
\end{eqnarray*}
where $\beta=\Phi(q)+\zeta_1=2\delta$ and $y=\beta^{-1}$. Denote
the first integral on the right-hand side $I_1$ and the second integral
$I_2$. After some algebra one sees that $U_{\epsilon_1,\infty}(x,s)$ equals
%
\begin{eqnarray}\label{forcomp}
&&\frac{e^s}{2}
\bigl(e^{\Phi(q)(x-\epsilon_1)}+e^{-\zeta
_1(x-\epsilon_1)} \bigr)-\frac{e^{\epsilon_1+k^*}\gamma}{\beta}
\bigl(e^{\Phi(q)(x-\epsilon_1)}-e^{-\zeta(x-\epsilon_1)} \bigr)
\nonumber
\\
&&\qquad{}-\frac{e^{\epsilon_1}}{2\beta}e^{-\zeta_1(x-\epsilon
_1)}I_1+\frac{e^{\epsilon_1}}{2\beta}e^{\Phi(q)(x-\epsilon
_1)}I_2
\\
&&\qquad{}+\frac{e^{\epsilon_1}}{2\beta}e^{\Phi(q)(x-\epsilon_1)}I_1-\frac
{e^{\epsilon_1}}{2\beta}e^{-\zeta_1(x-\epsilon_1)}I_2.
\nonumber
\end{eqnarray}
Next, note that the last line in (\ref{forcomp}) can be rewritten as
\begin{eqnarray*}
&&\frac{e^{\epsilon_1}}{2\beta} \bigl(e^{\Phi(q)(x-\epsilon
_1)}+e^{-\zeta_1(x-\epsilon_1)} \bigr)
(I_1-I_2)-\frac{e^{\epsilon
_1}}{2\beta}e^{-\zeta_1(x-\epsilon_1)}I_1+
\frac{e^{\epsilon
_1}}{2\beta}e^{\Phi(q)(x-\epsilon_1)}I_2
\\
&&\qquad=\frac{e^{\epsilon_1}}{2} \bigl(e^{\Phi(q)(x-\epsilon
_1)}+e^{-\zeta_1(x-\epsilon_1)} \bigr)
\bigl(e^{k^*}-e^{s-\epsilon_1}\bigr)
\\
&&\qquad\quad{}-\frac{e^{\epsilon_1}}{2\beta}e^{-\zeta_1(x-\epsilon
_1)}I_1+\frac{e^{\epsilon_1}}{2\beta}e^{\Phi(q)(x-\epsilon_1)}I_2,
\end{eqnarray*}
where the equality follows from evaluating $I_1-I_2$. Plugging this
into (\ref{forcomp}) and simplifying yields
\begin{eqnarray*}
U_{\epsilon_1,\infty}(x,s)&=&-e^{-\zeta_1(x-\epsilon_1)}e^{\epsilon
_1}\beta^{-1}I_1+e^{\Phi(q)(x-\epsilon_1)}e^{\epsilon_1}
\beta^{-1}I_2
\\
&&{}+e^{\epsilon_1+\Phi(q)(x-\epsilon_1)}e^{k^*}\beta^{-1}\zeta
_1+e^{\epsilon_1-\zeta_1(x-\epsilon_1)}e^{k^*}\beta^{-1}\Phi(q).
\end{eqnarray*}
Rearranging the terms completes the proof.
\end{pf*}
\end{appendix}

\section*{Acknowledgements}

I would like to thank A. E. Kyprianou and two anonymous referees for
their valuable comments which led to improvements in this paper.



\printaddresses

\end{document}